\documentclass[a4paper,12pt]{article}
\usepackage[latin1]{inputenc}
\usepackage[T1]{fontenc}
\usepackage{lmodern}

\usepackage{amsthm,amsmath,amsfonts,amssymb,bbm,mathrsfs}
\usepackage{mathtools}
\usepackage{enumitem}
\usepackage{url}
\usepackage{dsfont}
\usepackage{appendix}
\usepackage{amsthm}
\usepackage{color}
\usepackage{graphicx}

\usepackage{tikz}
\usetikzlibrary{decorations}
\usetikzlibrary{decorations.pathmorphing}
\usetikzlibrary{decorations.pathreplacing}
\usetikzlibrary{decorations.shapes}
\usetikzlibrary{decorations.text}
\usetikzlibrary{decorations.markings}
\usetikzlibrary{decorations.fractals}
\usetikzlibrary{decorations.footprints}

\usepackage[colorlinks=true, linkcolor=blue, urlcolor=black, citecolor=blue,pdfstartview=FitH]{hyperref}

\usepackage[english]{babel}

\usepackage{caption,tikz,subfigure}
\usetikzlibrary{shapes}
\usetikzlibrary{patterns}

\usepackage[top=2.7cm, bottom=2.7cm, left=2.4cm, right=2.4cm]{geometry}

\providecommand{\keywords}[1]
{
\small
\textbf{\textit{Keywords---}} #1

}

\makeatletter

\@addtoreset{equation}{section}
\makeatother

\setlist[enumerate,1]{label=(\roman*),font=\normalfont}

\newcommand{\cv}{\xrightarrow[t\to\infty]{}}
\newcommand{\cd}{\xrightarrow[t\to\infty]{\mathrm{(d)}}}

\newcommand{\cvd}{\xrightarrow[t\to\infty]{\mathrm{(vd)}}}
\newcommand{\cvdnot}{\xrightarrow[]{\mathrm{(vd)}}}

\newcommand{\cwdnot}{\xrightarrow[]{\mathrm{(wd)}}}
\newcommand{\e}{\mathrm{e}}
\newcommand{\es}{\mathbb{E}}
\newcommand{\de}{\coloneqq}
\newcommand{\ds}{\mathrm{d}s}
\newcommand{\dt}{\mathrm{d}t}
\newcommand{\dx}{\mathrm{d}x}
\newcommand{\g}{\mathcal G_{\beta,t}\otimes \mathcal G_{\beta',t}}
\newcommand{\G}{\mathcal G_{\beta,t}}
\newcommand{\Gp}{\mathcal G_{\beta',t}}
\newcommand{\GG}{\mathcal G_{\beta,t}^{\otimes 2}}

\newcommand{\gbg}{\mathcal G_{\beta,t}\otimes \mathcal G_{\beta',t}}
\newcommand{\gru}{\gamma_r(u)}
\newcommand{\grv}{\gamma_r(v)}
\newcommand{\gtu}{\gamma_t(u)}
\newcommand{\gtv}{\gamma_t(v)}
\newcommand{\ind}{\mathds{1}}

\newcommand{\norm}{\mathcal N(0,1)}
\newcommand{\nr}{\mathcal N_r}
\newcommand{\ns}{\mathcal N_s}
\newcommand{\nt}{\mathcal N_t}
\newcommand{\p}{\mathbb{P}}
\newcommand{\q}{Q(\beta,\beta')}
\newcommand{\qr}{Q^{\mathrm{REM}}(\beta,\beta')}
\newcommand{\re}{\mathbb{R}}
\newcommand{\sd}{\overset{\mathrm{(d)}}{=}}
\newcommand{\summ}{\sum_{u\in\nt}}
\newcommand{\sumt}{\sum_{u\in\nt^t} }

\newcommand{\sumw}{\sum_{w\in\nt}}
\newcommand{\sumd}{\sum_{u\in\nt(D)}}
\newcommand{\tds}{\xrightarrow[t\to\infty]{}}
\newcommand{\txt}{\tilde{x}(t)}
\newcommand{\txtt}{\tilde{x}^t(t)}
\newcommand{\txtu}{\tilde{x}_u(t)}
\newcommand{\U}{\mathcal U}

\theoremstyle{plain}
\newtheorem{thm}{Theorem}[section]
\newtheorem{prop}[thm]{Proposition}
\newtheorem{lem}[thm]{Lemma}

\theoremstyle{definition}

\theoremstyle{remark}
\newtheorem{rem}[thm]{Remark}

\usepackage{anyfontsize}

\date{January 4, 2022}
\title{The overlap distribution at two temperatures for the branching Brownian motion}
\author{
	Benjamin \textsc{Bonnefont}\thanks{Sorbonne Universit\'e, Laboratoire de Probabilit\'es Statistique et Mod\'elisation, LPSM, 75005 Paris, France. Email: \texttt{benjamin.bonnefont@upmc.fr}
}}

\begin{document}
	
\maketitle

\begin{abstract}
\noindent We study the overlap distribution of two particles chosen under the Gibbs measure at two temperatures for the branching Brownian motion. We first prove the convergence of the overlap distribution using the extended convergence of the extremal process obtained by Bovier and Hartung \cite{bovierhartung2017}. We then prove that the mean overlap of two points chosen at different temperatures is strictly smaller than in Derrida's random energy model. The proof of this last result is achieved with the description of the decoration point process obtained by A\"id\'ekon, Berestycki, Brunet and Shi \cite{abbs2013}. To our knowledge, it is the first time that this description is being used.
\end{abstract}


%{
%	\hypersetup{linkcolor=black}
%	\tableofcontents
%} 

%%%%%%%%%%%%%%%%%%%%%%%%%%%%%%%%%%%%%%

\section{Introduction}

%%%%%%%%%%%%%%%%%%%%%%%%%%%%%%%%%%%%%%

%%%%%%%%%%%%%%%%%%%%%%%%%%%%%%%%%%%%%%

\subsection{Overview}

%%%%%%%%%%%%%%%%%%%%%%%%%%%%%%%%%%%%%%

The (binary) branching Brownian motion (BBM) describes a system of particles that starts with a single one at $0$ which moves as a standard Brownian motion and splits into two new particles after a mean-one exponential time. These two particles then move independently according to Brownian motions and split with the same rule. It appears that the BBM belongs to a more general class of models, called log-correlated Gaussian fields, for which the dependence between the random variables starts to affect the extreme value statistics. In this class, one finds the branching random walk (BRW) and the two-dimensional discrete Gaussian free field (DGFF). Among all these models, it is commonly accepted that the BBM, with its branching structure and its brownian trajectories, is more manageable. We refer to \cite{arguin_2016} for more details on log-correlated fields.\\
\indent For $t\geq 0$, let $\mathcal N_t$ be the set of particles alive at time $t$ and let $x_u(t)$ denote the position of a particle $u\in\mathcal N_t$ at time $t$. One can interpret the position of the particles of the BBM as an energy, and in statistical physics, it is common to introduce the partition function and the corresponding Gibbs measure to study the extreme values. For $\beta>0$ and $t\geq0$, they are defined by
\[
Z_{\beta,t}\de\summ \e^{\beta x_u(t)},\hspace{20mm}
\G \de \frac{1}{Z_{\beta,t}} \summ \e^{\beta x_u(t)}\delta_u.
\]
One crucial quantity to investigate the energy landscape at the thermodynamical limit is the overlap $q_t(u,v)$ between the particles $u$ and $v$ and more specifically the following random probability measure
\[
\GG (q_t(u,v)\in \cdot),
\]
which is the distribution of the overlap between particles $u$ and $v$ chosen independently under the Gibbs measure $\G$ (see next subsection for precise definitions).
In their widely known article \cite{derridaspohn88}, Derrida and Spohn showed that this random measure converges to a limit whose support is $\{0,1\}$ in the low-temperature regime, exhibiting a one-step replica symmetry breaking (1-RSB) in the langage of spin glasses. Curiously enough, the limiting overlap behaves in the same way as in the i.i.d. case, the so-called random energy model (REM). This model was introduced by Derrida \cite{Derrida1981} as a toy-model to understand more complex spin glasses. As we will see, this is no longer the case when the temperatures are different.\\

In a very recent article, Derrida and Mottishaw \cite{DerridaMottishaw_2021} gave a thorough analysis of the overlap between two copies of the same REM at two temperatures.
The distribution of the overlap between two points sampled independently
according to Gibbs measures at temperatures $\beta$ and $\beta'$ is in our settings
\[
\gbg (q_t(u,v)\in \cdot).
\]
\noindent
This object originally appeared in the study of the temperature chaos problem, see for example \cite{Rizzo2009} for a survey. For the DGFF, Pain and Zindy \cite{PainZindy21} showed the convergence of the distribution of the two-temperature overlap and proved that its mean value is strictly smaller than the REM's, when the temperatures are different and below the critical temperature, contrary to the one-temperature case. In this paper, we are interested in the same questions but for the BBM. Even if this model is hierarchical and usually simpler to study, the difficulty here is to deal with the decoration process whose description is less explicit than for the DGFF.\\

Let us mention that there are two different approaches to tackle the limiting overlap at one temperature. The first one is to prove 1-RSB and then to recover Poisson-Dirichlet statistics thanks to Ghirlanda-Guerra identities, see \cite{arguinzindy2014}. This is the method used by Bovier and Kurkova \cite{bovierkurkova2004-2} for the BBM, by Arguin and Zindy \cite{arguinzindy2015} for the DGFF, and by Jagannath \cite{jagannath2016} for the BRW with Gaussian increments.
A second method is to use the now established convergence of the extremal process of these models towards a \textit{randomly shifted decorated Poisson point process}, see \cite{abbs2013} and \cite{abk2013} for the BBM, \cite{biskuplouidor2018} for the DGFF and \cite{madaule2017} for the BRW. It suggests a candidate for the limiting overlap and gives a strong support to prove the convergence. For instance, Mallein \cite{Mallein2018} proved the same result for the BRW using the convergence of the extremal process obtained by Madaule \cite{madaule2017}. However, it is not clear how the first of these approaches could be adapted in the multiple temperature case.\\

In this paper, we follow the same approach for the branching Brownian motion that Pain and Zindy \cite{PainZindy21} used for the DGFF. We show that the two-temperature overlap distribution converges to the one from the limiting process obtained by Bovier \& Hartung in \cite{bovierhartung2017} and that its expectation is strictly smaller than the REM's when the two temperatures are different.

%%%%%%%%%%%%%%%%%%%%%%%%%%%%%%%%%%%%%%
\subsection{Model and results}
%%%%%%%%%%%%%%%%%%%%%%%%%%%%%%%%%%%%%%

One way to construct a (binary) BBM is to realize it as a process decorating the infinite binary tree $\U=\cup_{n\in\mathbb Z_+}\{0,1\}^n$, with the convention $\{0,1\}^0=\{\varnothing\}$. We use $0$ and $1$ here instead of the classical Ulam-Harris notation because it will be more convenient for the definition of the genealogical embedding function $\gamma$ to come. For $u\in\U$, let $|u|$ be the length of $u$ and for $k\leq |u|$, let $u_k\in\{0,1\}$ be the $k$-th component of $u$. Write $u\leq v$ if $u$ is an ancestor of $v$ and $u\wedge v$ for the last common ancestor of $u$ and $v$.
For each $u\in\mathcal U$, let $b_u$ be the birth-time of $u$ and $d_u$ its death-time and for all $t\geq 0$,
let $\nt\de\{u\in\mathcal U: b_u\leq t<d_u \}$ be the set of particles alive at time $t$ and $x_u(t)$ the position of particle $u$ at time $t$. Ikeda et al. \cite{ikedanagasawawatanabe68a,ikedanagasawawatanabe68b,ikedanagasawawatanabe69} proved that there exists a probability space $(\Omega,\mathcal F,\mathbb P)$ such that the trajectories are Brownian motions and the underlying tree $T=(T_t)_{t\geq 0}$ is a binary continuous time Galton-Watson tree with branching rate $1$.\\

The position of the highest particle $x(t)\de\max_{u\in\nt} x_u(t)$ has been the subject of intense studies since McKean \cite{mckean75} who linked the distribution function of $x(t)$ with the F-KPP partial differential equation. Then, Bramson \cite{bramson78} obtained the right centering term $m_t\de\sqrt 2 t-(3/2\sqrt 2)\log t$ and Lalley and Sellke \cite{lalleysellke87} obtained an integral representation of the limiting law using the limiting \textit{derivative martingale} $Z:=\lim\limits_{t\to\infty}\summ (\sqrt2 t-x_u(t))\e^{-\sqrt2(\sqrt2 t-x_u(t))}$.
A new step has been taken with the proof of the convergence of the \textit{extremal process}
\[
\mathcal E_t\de\summ \delta_{\tilde{x}_u(t)},\qquad \text{where~~}  \tilde{x}_u(t)\de x_u(t)-m_t,
\]
in the space of Radon measures on $\re$ endowed with the vague topology, to a randomly shifted decorated Poisson point process, simultaneously in \cite{abbs2013} and \cite{abk2013}. More precisely,
\begin{equation}\label{cv}
\mathcal E_t \cvd \sum_{i,j}\delta_{p_i+\Delta^i_j},
\end{equation}
\noindent
where the $(p_i)_{i \geq 0}$ are the atoms of a Cox process on $\re$ with intensity measure $CZe^{-\sqrt 2x} \dx$, $C$ a positive constant, $Z$ is the limiting derivative martingale introduced before and $(\Delta^i)_{i \geq 0}$ are i.i.d. point processes on $\re_-$ called \textit{decoration processes}.\\
Here and after, the set of summation of index like $i$, $j$ or $k$ is assumed to be $\mathbb Z_+$ unless otherwise specified and one identifies a simple point process with the set of its atoms.
We also use, as above, the superscript $\mathrm{(vd)}$ for the convergence in distribution of random measures with respect to the vague topology and $\mathrm{(wd)}$ with respect to the weak topology, as in the setting of \cite[Chapter 4]{kallenberg2017}.
\\

Let us recall, from the previous section, the following quantities defined for $\beta>0$ and $t>0$
\begin{align}\label{zfg}
Z_{\beta,t}\de\summ \e^{\beta x_u(t)},\hspace{10mm}
f_{t}(\beta)\de\frac{1}{t}\,\es[\log(Z_{\beta,t})],\hspace{10mm}
\G \de \frac{1}{Z_{\beta,t}} \summ \e^{\beta x_u(t)}\delta_u.
\end{align}
There is a phase transition at $\beta_c\de\sqrt 2$ which is related to the asymptotic speed of the extremal particle: we have the following convergence for the (averaged) free energy, see e.g. \cite{derridaspohn88},
\begin{align} \label{free}
\lim_{t\to\infty} f_t(\beta) 
= f(\beta) \coloneqq  \left\{
\begin{array}{ll}
1 + (\beta/\beta_c)^2, & \text{if } \beta \leq \beta_c, \\
2 \beta/\beta_c, & \text{if } \beta > \beta_c.
\end{array}
\right.
\end{align}
The overlap between particles $u,v\in\nt$ is defined by 
\[
q_t(u,v)\de\frac{1}{t}d_{u\wedge v}=\frac{1}{t}\es[x_u(t) x_v(t)].
\]
\noindent
It is known since \cite{abk2011} that the overlap of extremal particles is either $0$ or $1$ at the limit $t\rightarrow\infty$. 
In the following theorem, we establish the convergence of its distribution according to the Gibbs measure at two temperatures. Below the critical temperature, the limiting distribution is the one obtained from the limit of the extremal process in (\ref{cv}): the overlap is $1$ when the same cluster is chosen and $0$ otherwise.
\begin{thm} \label{thm1}
	Let $\beta,\beta'>0.$\\
	\indent
	(i) If $\beta\leq\beta_c$ or $\beta'\leq\beta_c$ and $a\in(0,1)$,
	$$\gbg (q_t(u,v)\geq a)\cv 0, \qquad in ~L^1.$$
	
	(ii) If $\beta,\beta'>\beta_c$, and $a\in(0,1)$,
	$$\gbg (q_t(u,v)\geq a)\cd \q ,$$
	where
\begin{equation}\label{overlap1} 
\q\de\dfrac{\sum_{i}\big(\sum_j \e^{\beta(p_i+\Delta^i_j)}\big)\big(\sum_j \e^{\beta'(p_i+\Delta^i_j)}\big)}{\big(\sum_{i,j} \e^{\beta(p_i+\Delta^i_j)}\big)\big(\sum_{i,j} \e^{\beta'(p_i+\Delta^i_j)}\big)}.
\end{equation}
\end{thm}
\noindent
The Part (i) of the theorem is a consequence of the convergence of the free energy and a Gaussian integration by parts.
The proof of Part (ii) uses the convergence in Equation (\ref{cv}) and an additionnal information on the genealogy of the extremal particles, which is what Bovier and Hartung obtained in \cite{bovierhartung2017}.
The authors define the following function, in a more general setting than ours,
\[
\gru\de\sum_{\substack{v\leq u:\\b_v\leq r\phantom{}}} u_{|v|}\; \e^{-b_v}, \qquad u\in\nt, \quad r\leq t.
\]
This function encodes the genealogy of the particles on $\re_+$ in the following way. First, $\gamma_0(\varnothing)=0$, then each time a particle $u\in\nt$ with value $\gtu$ splits at time $t$, one of the children keeps the same value $\gtu$ and the other one takes the value $\gtu+\e^{-t}$. If $r\leq t$ and $u\in\nt$, $\gru$ is simply $\grv$ where $v\in\nr$ is the ancestor of $u$ alive at time $r$. This way when two particles originate from a recent split, their images by the function are close. We refer to \cite{bovierhartung2017} for more details about the function $\gamma$. Bovier and Hartung obtain the following joint convergence of $\mathcal E_t$ with $(\gtu)_{u\in\nt}$:
\begin{prop}[Bovier and Hartung \cite{bovierhartung2017}]\label{ext}
Let $\widetilde{\mathcal E}_t \de \summ \delta_{(\gtu,\tilde{x}_u(t))}$, then $\widetilde{\mathcal E}_t$ converges in vague distribution to
	\[
	\widetilde{\mathcal E}\de
	\sum_{i,j} \delta_{(q_i,p_i)+(0,\Delta^i_j)},
	\]
	where $(q_i , p_i)_i$ are the atoms of a Cox process on $\mathbb R_+\times \mathbb R$ with intensity measure $Z(\mathrm d \nu) \otimes Ce^{-\sqrt 2x} \mathrm d x$,  $Z(\mathrm d \nu)$ is a random measure on $\mathbb R_+$ such that $Z(\re_+)=Z$ and $C$ and $(\Delta^i)_i$ were introduced in Equation (\ref{cv}).

\end{prop}

It is then natural to compare the overlap with the one obtained at the limit for the REM, where the positions at time $t$ are given by $|\nt|$ independent Brownian motions and is studied in \cite{kurkova2003}. One notices the same effect of the decoration process as in \cite{PainZindy21} for the DGFF: the expected value of the overlap is strictly smaller than in the REM case. More precisely, let us define, for $\beta,\beta'>\beta_c$,
$$\qr\de\dfrac{\sum_i \e^{(\beta+\beta')\eta_i} }{\sum_i \e^{\beta\eta_i}\sum_i \e^{\beta'\eta_i}},$$
where the $(\eta_i)_i$ are the atoms of a PPP($\e^{-\sqrt 2x}\dx$).

\begin{thm}\label{thm2}
	 For any $\beta\neq\beta'>\beta_c$, we have $\es  [\q]<\es [ Q^{\mathrm{REM}}(\beta,\beta')].$
\end{thm}
\begin{rem}
	When $\beta=\beta'$, one has $Q(\beta,\beta)\overset{(\mathrm d)}{=}Q^{\mathrm{REM}}(\beta,\beta)$ and its expected value is $1-\frac{\beta_c}{\beta}$, see \cite{kurkova2003}.
\end{rem}

What we need to prove this result is the precise description of the decoration point process obtained in \cite{abbs2013} and some technical estimates. Let us emphasize here that the main difficulty is to handle the description of this point process, which is less explicit than for the DGFF.

%%%%%%%%%%%%%%%%%%%
\subsection{Organization of the paper}
%%%%%%%%%%%%%%%%%%%

In Section \ref{ovcv}, we prove successively Part (i) and Part (ii) of Theorem \ref{thm1}. Then, in Section \ref{ovsmaller}, we give a proof of Theorem \ref{thm2}. Appendix \ref{a} contains some technical results that are used in Section \ref{ovcv} and Appendix \ref{Bessel}, \ref{tech} contains technical results that are used in Section \ref{ovsmaller}.

%%%%%%%%%%%%%%%%%%%%%%%%%%%%%%%%%%%%%%
\section{Proof of Theorem \ref{thm1}: convergence of the overlap distribution}
\label{ovcv}
%%%%%%%%%%%%%%%%%%%%%%%%%%%%%%%%%%%%%%

%%%%%%%%%%%%%%%%%%%
\subsection{Proof of Part (i) of Theorem \ref{thm1}}\label{part1}
%%%%%%%%%%%%%%%%%%%
The averaged free energy $f_t$ defined in (\ref{zfg}) is a convex function of $\beta$ and its limit $f$ is differentiable everywhere. By an argument of convexity known as Griffiths' lemma, see for example  \cite[page 25]{talagrand2011}, the derivative $f'$ is the pointwise limit of $f_t'$
\[
f'(\beta)=\lim\limits_{t \to \infty}f'_t(\beta)=\lim\limits_{t\to\infty}\frac{1}{t}\es\bigg[\summ x_u(t)\frac{\e^{\beta x_u(t)}}{\sumw \e^{\beta x_w(t)}}\bigg].
\]
One would like to apply a Gaussian integration by parts to the last term to make appear the overlap (see Lemma \ref{IPP}). In order to deal with a fixed number of Gaussian variables, we use a conditioning on the underlying tree $T$. Indeed, conditionally on $T$, $(x_u(t),u\in\nt)$ is a Gaussian vector with covariances $(u\wedge v)_{u,v\in\nt}$ and denoting $\es[\cdot\,|\,T]$ by $\es_T$ yields
\begin{align*}
\es_T\bigg[\summ x_u(t)\frac{\e^{\beta x_u(t)}}{\sumw \e^{\beta x_w(t)}}\bigg]
&=
\sum_{u,v\in\nt} u\wedge v\,\es_T\bigg[\frac{-\beta\e^{\beta( x_u(t)+x_v(t))}}{(\sumw \e^{\beta x_w(t)})^2}\bigg]
+\summ t\,\es_T\bigg[\frac{\beta\e^{\beta x_u(t)}}{\sumw \e^{\beta x_w(t)}}\bigg]\\
&=\beta t\Big(1-\es_T\big[\GG(q_t(u,v))\big]\Big).
\end{align*}
Taking expectation of both sides and using $f'(\beta)=\beta$ for $\beta\leq\beta_c$, see Equation (\ref{free}), yields $\es[\GG(q_t(u,v))]\longrightarrow 0$, as $t\rightarrow\infty$. Now assume that $\beta,\beta'>0$ with $\beta\leq\beta_c$ without loss of generality and take $a\in(0,1)$, we have
\begin{align*}
\gbg(q_t(u,v)\geq a)&=\sum_{w\in\mathcal N_{at}}\G(u\in\nt : u\geq w)\Gp(v\in\nt : v\geq w)\\
&\leq\max_{w\in\mathcal N_{at}}\G(u\in\nt : u\geq w)\sum_{w\in\mathcal N_{at}}\Gp(v\in\nt : v\geq w)\\
&=\max_{w\in\mathcal N_{at}}\G(u\in\nt : u\geq w).
\end{align*}
This last term converges to $0$ in $L^2$ since
\[
\max_{w\in\mathcal N_{at}}\G(u\in\nt : u\geq w)^2= \max_{w\in\mathcal N_{at}}\GG(u,v\in\nt : u,v\geq w)\leq \GG(q_t(u,v)\geq a),
\]
concluding the proof.
%%%%%%%%%%%%%%%%%%%
\subsection{Convergence of $(\rho_{\beta,t})_t$}\label{onet}
\label{rhoc}
%%%%%%%%%%%%%%%%%%%

In this subsection, we study the convergence of the following random measures on $\re_+$ defined by
\[
\rho_{\beta,t}\de\summ \e^{\beta\tilde{x}_u(t)}\delta_{\gtu}, \qquad \text{for }  \beta>\beta_c.
\]
We prove that $\rho_{\beta,t}\cvdnot\rho_{\beta}$, when $t\rightarrow\infty$, where $\rho_{\beta}$ is the corresponding measure in the limiting process, namely
\[
\rho_{\beta}\de\sum_{i,j} \e^{\beta(p_i+\Delta^i_j)}\delta_{q_i}.
\]
\begin{rem}\label{alternative_exp}
	Observing that $Z(\mathrm d \nu) \otimes Ce^{-\sqrt 2x} \dx=Z(\mathrm d \nu)/Z \otimes e^{-\sqrt 2(x-\frac{1}{\sqrt2} \log(CZ))} \dx$, it is easy to see that the point process whose atoms are $\xi_k = p_k-\frac{1}{\sqrt 2}\log(CZ)$ is a PPP($\e^{-\sqrt2 x}\dx$) independent of $Z$, and that the $(q_i)$ are i.i.d. with distribution $Z(\cdot)/Z(\mathbb R_+)$ and independent from $(\xi_k)$. It yields to the following expression for $\rho_{\beta}$:
	\[
	\rho_{\beta} = (CZ(\re_+))^{\frac{\beta}{\beta_c}} \sum_{k} \e^{\beta\xi_k} (\sum_j \e^{\beta \Delta^k_j})\delta_{q_k},
	\]
	This form will be useful in Subsection \ref{func}.
\end{rem}

A natural idea to prove the convergence is to use the relation $\rho_{\beta,t}(f)=\widetilde{\mathcal E}_t (\tilde f)$, where $f\in C_c^+(\mathbb R_+)$ and $\tilde f(x,h)\de\e^{\beta h}f(x)$ together with the convergence of $(\widetilde{\mathcal E}_t)$, see Proposition \ref{ext}. The problem is that $\tilde f$ is no more compactly supported and we thus need to control the high and low values of $h$.\\
For this purpose, let us denote $ \nt(D)\de\{u\in \nt : \tilde{x}_u(t)\in D \}$,  for $D\subset \mathbb R$, and define

\begin{align}\label{rhod}
\rho^D_{\beta,t}\de\sumd  \e^{\beta\tilde{x}_u(t)}\delta_{\gtu},
\qquad
\rho_{\beta}^D\de\sum_{i,j}\ind_{D}(p_i+\Delta^i_j) \e^{\beta(p_i+\Delta^i_j)}\delta_{q_i}.
\end{align}
When $D=[-A,A]$, it is easy to see that $\rho^D_{\beta,t}\cvdnot\rho_\beta^D$ when $t\rightarrow\infty$. Indeed, for $\varepsilon>0$, choose continuous functions $\phi_\varepsilon$ and $\psi_\varepsilon$ such that $\ind_{[-A+\varepsilon,A-\varepsilon]}\leq\phi_\varepsilon\leq \ind_{[-A,A]}\leq\psi_\varepsilon\leq\ind_{[-A-\varepsilon,A+\varepsilon]}$. Let $f\in C_c^+(\re_+)$ and define $g_\varepsilon(x,h)\de f(x)\e^{\beta h}\phi_\varepsilon(h)$ and $h_\varepsilon(x,h)\de f(x)\e^{\beta h}\psi_\varepsilon(h)$, then we have
\[
\es[\exp(-\widetilde{\mathcal E}_t(h_\varepsilon))]\leq\es[\exp(-\rho^D_{\beta,t}(f))]\leq \es[\exp(-\widetilde{\mathcal E}_t(g_\varepsilon))].
\]
Since $g_\varepsilon$ and $h_\varepsilon$ are compactly supported, Proposition \ref{ext} yields
\[
\es[\exp(-\widetilde{\mathcal E}(h_\varepsilon))]\leq\liminf\limits_{t \to \infty}\es[\exp(-\rho^D_{\beta,t}(f))]\leq\limsup\limits_{t \to \infty}\es[\exp(-\rho^D_{\beta,t}(f))]\leq \es[\exp(-\widetilde{\mathcal E}(g_\varepsilon))],
\]
and the dominated convergence theorem, when $\varepsilon\rightarrow0$, gives
\[\lim\limits_{t \to \infty}\es[\exp(-\rho^D_{\beta,t}(f))]=\es[\exp(-\rho^D_{\beta}(f))],\]
and the convergence of the Laplace functionals concludes.\\

The end of this subsection shows that this convergence still holds for $D=\mathbb R$. It is again sufficient to prove the convergence of the Laplace functionals
\[
\es\big[\e^{-\rho_{\beta,t}(f)}\big]\tds\es\big[\e^{-\rho_{\beta}(f)}\big], \qquad \forall f\in C_c^+(\mathbb R_+).
\]
Let $D=[ -A,A ]$ for $A\geq0$. Since $\rho^D_{\beta,t}(f) \leq \rho_{\beta,t}(f)$, we have
$\limsup\limits_{t \to \infty}\es[\e^{-\rho_{\beta,t}(f)}]\leq \es[\e^{-\rho^D_{\beta}(f)}]$ and the dominated convergence theorem, when $A\rightarrow\infty$, yields $\limsup\limits_{t \to \infty}\es[\e^{-\rho_{\beta,t}(f)}]\leq \es[\e^{-\rho_{\beta}(f)}].$\\
The proof of $\underset{t\rightarrow \infty}{\mathrm{liminf~}}\es[\e^{-\rho_{\beta,t}(f)}]\geq \es[\e^{-\rho_{\beta}(f)}]$ needs more work. We need to estimate the  density of particles outside $D=[-A,A]$. The top values can be controlled by the following inequality, which can be found in the seminal work by Bramson \cite[Proposition 3]{bramson78}, where $\txt\de\max_{u\in\nt} \tilde{x}_u(t)$.

\begin{prop}[Bramson \cite{bramson78}]\label{Bram} There exists $c>0$ such that
\[
\mathbb P(\txt>A) \leq c(A+1)^2 \e^{-\sqrt 2 A}, \qquad \mathrm{for~} A\geq0, t\geq A^2 \mathrm{~and~} t\geq 2.
\]
\end{prop}

\noindent
To address the low values, we use the following proposition whose proof is postponed to the Appendix, see \ref{a}.
\begin{prop}\label{pro1} Let $\eta>0$, then $\underset{A\rightarrow \infty}{\mathrm{lim}} \underset{t\rightarrow \infty}{\mathrm{limsup~}} \p \big(\rho_{\beta,t}^{]-\infty,-A]}(\mathbb R_+)>\eta \big)=0.$
\end{prop}

\noindent
These two propositions enable to prove the lower bound.
\begin{prop}\label{liminf}
	Let $f\in C^+_c(\mathbb R_+)$, then
$\underset{t\rightarrow \infty}{\mathrm{liminf~}}\es[\e^{-\rho_{\beta,t}(f)}]\geq \es[\e^{-\rho_{\beta}(f)}]$.
\end{prop}

\begin{proof}
Fix $\eta>0$, we have
\begin{align*}
\p \big(\rho^{[-A,A]^c}_{\beta,t}(\mathbb R_+)>\eta \big)
&\leq \p \big(\rho^{]-\infty,-A]}_{\beta,t}(\mathbb R_+)>\eta\big)+ \p(\txt>A).
\end{align*}
Applying Proposition \ref{pro1} for the first term and Proposition \ref{Bram} for the second one yields
\begin{equation}\label{rho}
\underset{A\rightarrow \infty}{\mathrm{lim}} \underset{t\rightarrow \infty}{\mathrm{limsup~}} \p \big(\rho^{[-A,A]^c}_{\beta,t}(\mathbb R_+)>\eta \big)=0.
\end{equation}
\noindent
The fact that $f$ is bounded implies 
\[
\underset{A\rightarrow \infty}{\mathrm{lim}} \underset{t\rightarrow \infty}{\mathrm{limsup~}} \p \big(\rho^{[-A,A]^c}_{\beta,t}(f)>\eta \big)=0.
\]
Now fix $\varepsilon>0$ and let $A_0,t_0$ be large enough such that 
$\mathbb P (\rho^{[-A_0,A_0]^c}_{\beta,t}(f)>\eta) <\varepsilon$ for $t\geq t_0$.
Then, with $D=[-A_0,A_0]$, we obtain, for $t\geq t_0$,
\begin{align*}
\es\big[\e^{-\rho_{\beta,t}(f)}\big]& = \es\big[\e^{-\rho^D_{\beta,t}(f)}\e^{-\rho^{D^c}_{\beta,t}(f)}\big]
 \geq \e^{-\eta}\es\big[\e^{-\rho^D_{\beta,t}(f)}\ind_{\rho^{D^c}_{\beta,t}(f)\leq\eta}\big]
 \geq \e^{-\eta}\big(\es\big[\e^{-\rho^D_{\beta,t}(f)}\big]-\varepsilon\big),
\end{align*}
which implies
\begin{align*}
\underset{t\rightarrow \infty}{\mathrm{liminf~}}\es\big[\e^{-\rho_{\beta,t}(f)}\big]& \geq \e^{-\eta}\big(\es\big[\e^{-\rho^D_{\beta}(f)}\big]-\varepsilon\big) \geq \e^{-\eta}\big(\es\big[\e^{-\rho_{\beta}(f)}\big]-\varepsilon\big).
\end{align*}
The last inequality is true for every $\varepsilon,\eta>0$ and the conclusion follows.
\end{proof}
\noindent

In order to prove Part (ii) of Theorem \ref{thm1}, we will need a bit more than the convergence for the vague topology.
\begin{prop}\label{mass}
	When $t\rightarrow\infty$, $\rho_{\beta,t}(\mathbb R_+)\xrightarrow[]{\mathrm{(d)}} \rho_{\beta}(\mathbb R_+)$.
\end{prop}

\begin{proof}
	We use Laplace transform again and the argument is very similar to the previous one: we just write down the liminf part. Given $\eta,\varepsilon>0$ and $\lambda>0$, there exists, by (\ref{rho}), $A_0,t_0>0$ such that $\mathbb P (\lambda\rho^{[-A,A]^c}_{\beta,t}(\mathbb R_+)>\eta) <\varepsilon$ for every $t\geq t_0$ and $A\geq A_0$. This gives
	\[
	\es\big[\e^{-\lambda\rho_{\beta,t}(\re_+)}\big]
	\geq \e^{-\eta}\big(\es\big[\e^{-\lambda\rho^{[-A,A]}_{\beta,t}(\re_+)}\big]-\varepsilon\big), \qquad \text{ for } t\geq t_0 \text{ and } A\geq A_0.
	\]
	\noindent
	This time, using the fact that $\rho^{[-A,A]}_{\beta,t}(\re_+)=\mathcal E_t(\exp(\beta \, \cdot)\ind_{[-A,A]})$ together with the convergence of the extremal process, see Equation (\ref{cv}), concludes the proof of Proposition \ref{mass}.	

\end{proof}
\noindent	
This last result shows that in fact the convergence in distribution $\rho_{\beta,t}\longrightarrow \rho_{\beta}$ holds for the weak topology by \cite[Theorem 4.19]{kallenberg2017}. 
\begin{rem}\label{extensionf}	
	The arguments of Proposition \ref{mass} show that the convergence in distribution of $\mathcal{E}_t(f)$ to $\mathcal E(f)$ holds for continuous function on $\mathbb R$ that are $\mathcal O(\exp(\alpha x))$ as $x\rightarrow -\infty$ for some $\alpha>\sqrt 2$, and without 	
	any restriction on the behaviour at $+\infty$.
\end{rem}

\begin{rem}
	Following Remark \ref{alternative_exp} and using the fact that $\sum_k\e^{\beta\xi_k}\sum_{j}\e^{\beta\Delta_j^k}$ has the same distribution as $\e^{\beta c_\beta}\sum_k\e^{\beta\xi_k}$ with $c_\beta\de\beta_c^{-1}\log\es[\e^{\beta_c X_\beta}]$, see Subsection \ref{func}, the previous convergence can be alternatively stated
	\[
	\summ \e^{\beta\tilde{x}_u(t)}
	\cd
	 Z(\re_+)^{\frac{\beta}{\beta_c}}\e^{\beta c_\beta}\sum_{k}\e^{\beta\xi_k},
	\]
	where the serie in the right term has a stable law.
\end{rem}

%%%%%%%%%%%%%%%%%%%
\subsection{Convergence of $(\rho_{\beta_1,t}\otimes \rho_{\beta_2,t})_t$ }
\label{twot}
%%%%%%%%%%%%%%%%%%%

Recall that  
$$ \rho_{\beta}=\sum_{i,j} \e^{\beta(p_i+\Delta^i_j)}\delta_{q_i}, \qquad \forall \beta>\beta_c.$$
\begin{prop}\label{wdconvergence}
For $\beta_1,\beta_2>\beta_c$, we have $\rho_{\beta_1,t}\otimes \rho_{\beta_2,t}\cwdnot\rho_{\beta_1}\otimes \rho_{\beta_2}$, when $t\rightarrow\infty$.

\end{prop}

\begin{proof}
Let us first prove the convergence for the vague topology: it is sufficient to prove that for any $f_1,f_2\in C_c^+(\mathbb R_+)$,
$$\es[\e^{-\rho_{\beta_1,t}(f_1)-\rho_{\beta_2,t}(f_2)}]\longrightarrow \es[\e^{-\rho_{\beta_1}(f_1)-\rho_{\beta_2}(f_2)}], \qquad \text{as } t\rightarrow\infty.$$
As in the beginning of Subsection \ref{onet}, a direct consequence of the extended convergence of the extremal process is the following convergence, for any  $D=[-A,A]$,
$$\rho^D_{\beta_1,t}\otimes \rho^D_{\beta_2,t}\cvdnot\rho^D_{\beta_1}\otimes \rho^D_{\beta_2}, \qquad \text{as } t\rightarrow\infty,$$
where $\rho^D_{\beta,t}$ and $\rho^D_{\beta}$ were defined in (\ref{rhod}).
From
\[
\rho^D_{\beta_1,t}(f_1)+\rho^D_{\beta_2,t}(f_2)  \leq \rho_{\beta_1,t}(f_1)+\rho_{\beta_2,t}(f_2) ,
\]
we deduce, as in the proof of Proposition \ref{liminf}, that
\[ \underset{t\rightarrow \infty}{\mathrm{limsup~}}\es[\e^{-\rho_{\beta_1,t}(f_1)-\rho_{\beta_2,t}(f_2)}]\leq \es[\e^{-\rho_{\beta_1}(f_1)-\rho_{\beta_2}(f_2)}].
\]
And it follows from Equation \eqref{rho} that
$$\underset{A\rightarrow \infty}{\mathrm{lim}} \underset{t\rightarrow \infty}{\mathrm{limsup~}} \mathbb P\big(\rho^{[-A,A]^c}_{\beta_1,t}(f_1)+\rho^{[-A,A]^c}_{\beta_2,t}(f_2)>\eta \big)=0.$$
\noindent
Fix $\varepsilon>0$ and let $A_0,t_0$ be large enough such that 
$\p\big(\rho^{[-A_0,A_0]^c}_{\beta_1,t}(f_1)+\rho^{[-A_0,A_0]^c}_{\beta_2,t}(f_2)>\eta\big) <\varepsilon$ for $t\geq t_0$.
Then, using $D=[-A_0,A_0]$, we obtain, for $t\geq t_0$,
\begin{align*}
\es[\e^{-\rho_{\beta_1,t}(f_1)-\rho_{\beta_2,t}(f_2)}]& = \es[\e^{-\rho^D_{\beta_1,t}(f_1)-\rho^D_{\beta_2,t}(f_2)}\e^{-\rho^{D^c}_{\beta_1,t}(f_1)-\rho^{D^c}_{\beta_2,t}(f_2)}]\\
& \geq \es[\e^{-\rho^D_{\beta_1,t}(f_1)-\rho^D_{\beta_2,t}(f_2)}\e^{-\eta}\ind_{\rho^{D^c}_{\beta_1,t}(f_1)+\rho^{D^c}_{\beta_2,t}(f_2)\leq\eta}]\\
& \geq \big(\es[\e^{-\rho^D_{\beta_1,t}(f_1)-\rho^D_{\beta_2,t}(f_2)}]-\varepsilon\big)\e^{-\eta},
\end{align*}
and the conclusion follows in the same way as in the proof of Proposition \ref{liminf}.\\
	In order to obtain the convergence for the weak topology, it is sufficient, thanks to  \cite[Theorem 4.19]{kallenberg2017}, to establish
	\[\rho_{\beta_1,t}(\re_+)\rho_{\beta_2,t}(\re_+)\cd\rho_{\beta_1}(\re_+)\rho_{\beta_2}(\re_+).
	\]
	For $\lambda,\mu>0$, the following equality
	\[
	\es[\e^{-\lambda\rho_{\beta_1,t}(\re_+)-\mu\rho_{\beta_2,t}(\re_+)}]=\es[\e^{-\mathcal E_t(\lambda \exp(\beta_1\cdot)+\mu\exp(\beta_2\cdot))}],
	\]
	with Remark \ref{extensionf} prove the joint convergence of $(\rho_{\beta_1,t}(\re_+),\rho_{\beta_2,t}(\re_+)) $ towards $(\rho_{\beta_1}(\re_+),\rho_{\beta_2}(\re_+))$, concluding the proof.

\end{proof}

%%%%%%%%%%%%%%%%%%%
\subsection{Proof of Part (ii) of Theorem \ref{thm1}}\label{ov}
%%%%%%%%%%%%%%%%%%%

The first step of the proof is to show that the convergence of $\mathcal G_{\beta,t}\otimes \mathcal G_{\beta',t} (q_t(u,v)\geq a)$ can be handled with $\gbg(|\gtu-\gtv|\leq \delta)$: this is the content of the following three propositions. Then we will use the convergence of $(\rho_{\beta,t}\otimes \rho_{\beta',t})_t$ from the previous subsection with the expression
\[
\gbg (|\gtu-\gtv|\leq \delta) = \dfrac{\rho_{\beta,t}\otimes \rho_{\beta',t}(\Delta^\delta)}{\rho_{\beta,t}\otimes \rho_{\beta',t}(\mathbb R_+^2)}, \qquad \text{ where } \Delta^\delta \de \{(x,y)\in \mathbb R_+^2 : |x-y|\leq \delta\}.
\]
Recall that $\nt(D)=\{u\in \nt : \tilde{x}_u(t)\in D \}$.
\begin{prop}\label{>a} Let $\delta>0$, $a\in(0,1)$ and $D\subset\mathbb R$ compact, then
	$$\underset{t\rightarrow \infty}{\mathrm{lim}}\p \big(\exists u,v \in \nt(D) : q_t(u,v)\geq a \mathrm{~and~} |\gtu - \gtv|>\delta\big)=0.$$
\end{prop}
\begin{proof}
Fix $\varepsilon>0$. For $r<t$, define, as in \cite[Lemma 4.1]{bovierhartung2017}, the events
\[
\mathcal A^\gamma_{r,t}(D) =\{ \forall u \in \nt(D) : \gtu-\gru \leq \e^{-r/2} \}.
\]
The same lemma gives the existence of $r(D,\varepsilon)\geq 0$ such that $\p(\mathcal A^\gamma_{r,t}(D))\geq 1-\varepsilon$ whenever $r>r(D,\varepsilon)$ and $t>3r$. Pick $r>0$ such that $2\e^{-r/2}<\delta$ and $r>r(D,\varepsilon)$. If $t>\max\{3r,\frac{r}{a}\}$, we have, on $\mathcal A^\gamma_{r,t}(D)$,
\[
\gtu-\gru \leq \frac{\delta}{2},\qquad \forall u \in\nt(D).
\]
Then observe that $q_t(u,v)\geq a$ implies that the trajectories of $u$ and $v$ coincide at least up to time $at>r$ so that $|\gtu - \gtv|\leq \delta$, on $\mathcal A^\gamma_{r,t}(D)$.
Therefore
$\{\exists u,v \in \nt(D) : q_t(u,v)\geq a \mathrm{~and~} |\gtu- \gtv|>\delta\}\subset \mathcal A^\gamma_{r,t}(D)^c$,
which has a probability smaller than $\varepsilon$.
\end{proof}
\noindent
The following proposition is direct consequence of Theorem 2.1 in \cite{abk2011} and it shows that the choice of $a\in(0,1)$ has no effect, the overlap being concentrated on $0$ and $1$ at the limit.

\begin{prop}[Arguin, Bovier and Kistler \cite{abk2011}]\label{01}
	For any compact set $D\subset \mathbb R$ and $a\in(0,1/2)$,
\[
\underset{t\rightarrow\infty}{\mathrm{lim}}~ \p \big(\exists u,v \in \nt(D) : q_t (u,v) \in (a,1 - a)\big) = 0.
\]
\end{prop}
\noindent
A last result is needed.

\begin{prop}\label{<a} Let $a\in(0,1)$ and $D\subset\mathbb R$ compact, we have
	\[
	\underset{\delta\rightarrow 0}{\mathrm{lim~}}\underset{t\rightarrow \infty}{\mathrm{limsup~}}\p \big(\exists u,v \in \nt(D) : q_t(u,v)< a \mathrm{~and~} |\gtu-\gtv|\leq\delta\big)=0.
	\]
\end{prop}
\begin{proof}
	Thanks to Proposition \ref{01} and by monotonicity in $a$, we can assume that $a=\frac{1}{4}$. Given $\varepsilon>0$, Lemma 4.3 in \cite{bovierhartung2017} provides $\delta>0$ and $r(\delta,\varepsilon)$ such that for any $r>r(\delta,\varepsilon)$ and $t>3r$
	\[
	\p\big(\exists u,v \in \nt(D) : q_t (u, v) \leq \frac{r}{t} \mathrm{~and~} |\gtu-\gtv|\leq\delta\big) <\varepsilon.
	\]
	Taking $t=4r$ concludes the proof.
\end{proof}

\begin{proof}[Proof of Theorem \ref{thm1}]
Let us denote $A_{\delta,t} \de\{(u,v)\in\nt^2 : |\gtu-\gtv|\leq\delta \}$, we have
\begin{eqnarray*}
&&\g(q_t(u,v)\geq a)-\g(A_{\delta,t})
\\
&=&\g(q_t(u,v)\geq a;A_{\delta,t}^c) -\g(q_t(u,v)< a; A_{\delta,t}).
\end{eqnarray*}
Then, with $D=[-A,A]$, one gets
\begin{eqnarray*}
&&|\g(q_t(u,v)\geq a)-\g(A_{\delta,t})|
\\
&\leq& \g(q_t(u,v)\geq a; A_{\delta,t}^c\cap\nt(D)^2)
+\g(q_t(u,v)<a;A_{\delta,t}\cap\nt(D)^2)
\\
&&+\g((u,v)\notin \nt(D)^2)
\\
&\leq& \ind_{\{\exists u,v\in\nt(D) \,:\, q_t(u,v)\geq a, |\gtu-\gtv|>\delta\}}
+ \ind_{\{\exists u,v\in\nt(D) \,:\, q_t(u,v)< a, |\gtu-\gtv|\leq\delta\}}
\\
&&+\mathcal G_{\beta,t}(\nt(D^c)) + \mathcal G_{\beta',t}(\nt(D^c)).
\end{eqnarray*}
We thus have, for every $\eta>0$,
\begin{eqnarray*}
	&&\p(|\g(q_t(u,v)\geq a)-\g(A_{\delta,t})|>\eta)
	\\
	&\leq&\p(\exists u,v\in\nt(D) : q_t(u,v)\geq a, |\gtu-\gtv|>\delta)+\\
	&&\p(\exists u,v\in\nt(D): q_t(u,v)< a, |\gtu-\gtv|\leq\delta)+\p(\mathcal G_{\beta,t}(\nt(D^c)) + \mathcal G_{\beta',t}(\nt(D^c))>\eta/3).
\end{eqnarray*}

\noindent The first term can be handled with Proposition \ref{>a}, the second one with Proposition \ref{<a} and the following lemma, whose proof is postponed to Appendix \ref{a}, deals with the third term.
\begin{lem} Let $\eta>0$, then
	$\lim\limits_{A\to \infty} \limsup\limits_{t\to\infty} \p(\mathcal G_{\beta,t}(\nt([-A,A]^c))>\eta)=0$.
\end{lem}
\noindent
We finally obtain:
\begin{equation}\label{adelta}
\lim\limits_{\delta \to 0} \limsup\limits_{t \to \infty} \p (|\g(q_t(u,v)\geq a)-\g(|\gtu-\gtv|\leq\delta)|>\eta)=0.
\end{equation}
\noindent
Now for every $\delta>0$, choose a continuous function $f_\delta$ on $\mathbb R_+^2$ such that $\ind_{\Delta^\delta}\leq f_\delta \leq \ind_{\Delta^{2\delta}}$. We have
$$\dfrac{\rho_{\beta,t}\otimes \rho_{\beta',t}(f_{\delta/2})}{\rho_{\beta,t}\otimes \rho_{\beta',t}(\mathbb R_+^2)} \leq\g(|\gtu-\gtv|\leq\delta)\leq \dfrac{\rho_{\beta,t}\otimes \rho_{\beta',t}(f_\delta)}{\rho_{\beta,t}\otimes \rho_{\beta',t}(\mathbb R_+^2)}.$$
Then, for $\lambda>0$, the convergence from Proposition (\ref{wdconvergence}) gives
\begin{align*}
\es\Big[\exp\big(-\lambda\frac{\rho_{\beta}\otimes \rho_{\beta'}(f_{\delta})}{\rho_{\beta}\otimes \rho_{\beta'}(\mathbb R_+^2)}\big)\Big] &\leq\liminf\limits_{t \to \infty}\es[\exp\{-\lambda\g(|\gtu-\gtv|\leq\delta)\}],
\end{align*}
and
\begin{align*}
\limsup\limits_{t \to \infty}\es[\exp\{-\lambda\g(|\gtu-\gtv|\leq\delta)\}]&\leq \es\Big[\exp\big(-\lambda\frac{\rho_{\beta}\otimes \rho_{\beta'}(f_{\delta/2})}{\rho_{\beta}\otimes \rho_{\beta'}(\mathbb R_+^2)}\big)\Big].
\end{align*}
The dominated convergence theorem with $\delta\rightarrow0$ and Equation \eqref{adelta} yield
$$\lim\limits_{t \to \infty} \es\Big[\exp\{-\lambda\g(q_t(u,v)\geq a)\}\Big]= \es\Big[\exp\big(-\lambda \frac{\rho_{\beta}\otimes \rho_{\beta'}(\Delta)}{\rho_{\beta}\otimes \rho_{\beta'}(\mathbb R_+^2)}\big)\Big].$$
Finally, observe that almost surely
\[
\frac{\rho_{\beta}\otimes \rho_{\beta'}(\Delta)}{\rho_{\beta}\otimes \rho_{\beta'}(\mathbb R_+^2)}=\q,
\]
since $Z(\mathrm d\nu)$ is a.s. non-atomic  by \cite[Proposition 3.2]{bovierhartung2017}. This concludes the proof of Part (ii) of Theorem \ref{thm1}.
\end{proof}

%%%%%%%%%%%%%%%%%%%%%%%%%%%%%%%%%%%%%%

\section{Proof of Theorem \ref{thm2}: the mean overlap is smaller than the REM's}\label{ovsmaller}

%%%%%%%%%%%%%%%%%%%%%%%%%%%%%%%%%%%%%%

The aim of this section is to prove Theorem \ref{thm2}. In Subsection \ref{func}, we notice that it is sufficient to prove that a certain functional of the decoration can take arbitrary small values. Then, in Subsection \ref{lapmet}, we use the description of the decoration process obtained in \cite{abbs2013} to study the support of this functional.

%%%%%%%%%%%%%%%%%%%
\subsection{A functional of the decoration}
\label{func}
%%%%%%%%%%%%%%%%%%%

Let $\mathcal C\de\sum_{j\geq0}\delta_{\Delta_j}$ be a point process distributed as the decoration process arising in (\ref{cv}). We assume that its atoms are ranked in non-increasing order, so that $\Delta_0=0$ and $\Delta_j\leq 0$ for all $j\geq1$. Denote, for $\beta>\beta_c$, $X_\beta \de\frac{1}{\beta}\log \sum_{j\geq0} \e ^{\beta \Delta_j}$, which is well defined a.s. thanks to the following result.
\begin{lem}
	For every $\beta>\beta_c$, we have $\es[\sum_{j\geq0} \e ^{\beta \Delta_j}]<\infty.$
\end{lem}
\begin{proof}
	Observe that the following integral representation holds:
	$$\sum_{j\geq0} \e ^{\beta \Delta_j} = \beta \int_{0}^{\infty} \mathcal C([-s,0])\e^{-\beta s}\mathrm ds. $$
	And Fubini-Tonelli theorem gives
	$$\es\Big[\sum_{j\geq0} \e ^{\beta \Delta_j}\Big] = \beta \int_{0}^{\infty} \es\big[\mathcal C([-s,0])\big]\e^{-\beta s}\mathrm ds. $$
	Proposition 1.5 in \cite{cortineshartunglouidor2019} gives the asymptotic  $\es\big[\mathcal C([-s,0])\big]\sim C_{\star}\e^{\beta_c s}$ as $s\rightarrow\infty$, for some $C_\star>0$, concluding the proof.
\end{proof}
\noindent
Recall the expression of the overlap from (\ref{overlap1})
\begin{equation}\label{overlap} 
\q\de\dfrac{\sum_{i}\big(\sum_j \e^{\beta(p_i+\Delta^i_j)}\big)\big(\sum_j \e^{\beta'(p_i+\Delta^i_j)}\big)}{\big(\sum_{i,j} \e^{\beta(p_i+\Delta^i_j)}\big)\big(\sum_{i,j} \e^{\beta'(p_i+\Delta^i_j)}\big)}, \qquad \beta,\beta'>\beta_c.
\end{equation}
Introducing $X_{\beta,i}\de\frac{1}{\beta}\log \sum_{j\geq0} \e ^{\beta \Delta^i_j}$ and $\xi_i=p_i-\frac{1}{\sqrt 2}\log(CZ)$ which form a PPP($\e^{-\sqrt 2x}\dx$), see Remark \ref{alternative_exp}, the overlap can be rewritten in the following manner:
\[
\q =\dfrac{\sum_i \e^{\beta(\xi_i+X_{\beta,i})} \e^{\beta'(\xi_i+X_{\beta',i})}}{\big(\sum_i \e^{\beta(\xi_i+X_{\beta,i})}\big)\big(\sum_i \e^{\beta'(\xi_i+X_{\beta',i})}\big)}.
\]
Then, Lemma 2.1 in \cite{panchenkotalagrand2007-1} shows that the point process $(\xi_i+X_{\beta,i},\xi_i+X_{\beta',i})_{i}$ has the same distribution as $(\xi_i+c_\beta,\xi_i+c_\beta+Y_i)_{i}$ where $(Y_i)_{i\geq1}$ are i.i.d. and independent of $(\xi_i)_{i\geq1}$, and $c_\beta\de\beta_c^{-1}\log\es[\e^{\beta_c X_\beta}]$, so that
\[
\q \sd\dfrac{\sum_i \e^{\beta\xi_i} \e^{\beta'(\xi_i+Y_i)}}{\big(\sum_i \e^{\beta\xi_i}\big)\big(\sum_i \e^{\beta'(\xi_i+Y_i)}\big)}.
\]
In order to prove that $\es[\q]<\es[\qr]$ when $\beta\neq\beta'$, we stick to the strategy of Pain and Zindy in \cite[Section 3]{PainZindy21} with the following lemma which shows that the $Y_i$ play a negative role in the expected value of the overlap.

\begin{lem}[Pain and Zindy \cite{PainZindy21}] \label{lemPZ}
	Let $(p_n)_{n\geq 1}$ and $(q_n)_{n\geq 1}$ be nonincreasing deterministic sequences of nonnegative real numbers such that $\sum_{n\geq 1} p_n = 1$.
	Let $(A_n)_{n\geq 1}$ be a sequence of i.i.d.\@ positive random variables.
	We set 
	\[
	\tilde{p}_n \coloneqq \frac{A_n p_n}{\sum_{k\geq 1} A_k p_k}, 
	\quad \forall n\geq 1.
	\]
	Then, we have 
	\begin{equation}
	\es\Big[{\sum_{n\geq 1} \tilde{p}_n q_n}\Big] 
	\leq \sum_{n\geq 1} p_n q_n. \label{ineq}
	\end{equation}
	Moreover, if $A_1$ is not almost surely constant, $(q_n)_{n\geq 1}$ is not constant and, for any $n \geq 1$, $p_n > 0$, then the inequality in \eqref{ineq} is strict.
\end{lem}
\noindent
Assuming the $(\xi_k)$ are ranked in decreasing order and using the previous lemma, conditionally on $(\xi_k)$, with 
\[
p_n \coloneqq 
\frac{\e^{\beta' \xi_n}}{\sum_{k\geq 1} \e^{\beta' \xi_k}},
\quad
q_n \coloneqq 
\frac{\e^{\beta \xi_n}}{\sum_{k\geq 1} \e^{\beta \xi_k}},
\quad \text{and} \quad
A_n \coloneqq
\e^{\beta' Y_n},
\]
shows that $\es[\q]<\es[\qr]$ on the condition that $Y_1$ is not almost surely constant.\\
The final step is to prove that $Y_1$ is not constant almost surely. Assume that it is not the case, then $X_\beta - X_{\beta'}$ is also constant almost surely and therefore there exists $c>0$ such that, almost surely
\begin{equation}\label{lpeq}
\Big(\sum_{j\geq0} \e ^{\beta \Delta_j}\Big)^{\frac{1}{\beta}}=c\,\Big(\sum_{j\geq0} \e ^{\beta' \Delta_j}\Big)^{\frac{1}{\beta'}}.
\end{equation}
\noindent
If we can prove that $\sum_{j\geq 1} \e ^{\beta \Delta_j}$ can be arbitrary small with positive probability, the following lemma concludes the proof of Theorem \ref{thm2} by showing that (\ref{lpeq}) is impossible.
\begin{lem}
	Let $\alpha>1$, $c>0$ and $(a_{k,j})_{k,j\geq 1}$ be non-negative real numbers such that $\sum_{j\geq 1} a_{k,j}<\infty$. Assume that, for all $k\geq 1$,
	$$ 1+\sum_{j\geq 1} a_{k,j}=c\,\Big(1+\sum_{j\geq 1} a_{k,j}^{\alpha}\Big)^{\frac{1}{\alpha}} \qquad and \qquad
	\sum_{j\geq 1} a_{k,j} \underset{k\rightarrow \infty}{\longrightarrow} 0.$$
	Then $c=1$ and all the $a_{k,j}$ are equal to 0.
\end{lem}
\begin{proof}
	For a non-negative sequence $x=(x_j)_j$, let us denote $|x|_p=(\sum_{j\geq 1} x_j ^p)^{1/p}$ for $p\geq1$ when the sum is finite. The assumption $|a_k|_1\longrightarrow 0$, when $k\rightarrow \infty$, implies that $ |a_{k}|_{\alpha}\longrightarrow 0$, when $k\rightarrow \infty$.
	Taking the limit $k$ tends to infinity in the equality yields $c=1$. The second assertion is a consequence of the fact: $|x|_p=|x|_{p'}$ for $p \neq p' \Rightarrow x$ is zero except perhaps on one point.
\end{proof}	
Next section is devoted to the last result we need for the proof of Theorem \ref{thm2}.
\begin{prop}\label{sup}
	Let $\beta>\beta_c$, then $0$ is in the support of the law of $\sum_{j\geq 1}\e^{\beta\Delta_j}$.
\end{prop}

%%%%%%%%%%%%%%%%%%%
\subsection{Laplace transform along the backward path $Y$}
\label{lapmet}
%%%%%%%%%%%%%%%%%%%

In order to prove Proposition \ref{sup}, we will use the description of the decoration process $\mathcal Q$ obtained in \cite{abbs2013}. It is obtained with conditioned branching Brownian motions issued from a certain path $Y$, the \textit{backward path}, which we will explicit now. We refer to the article for more details.
The authors adopt a different normalization in their article: each of the particles follows a Brownian motion with drift $2$ and variance $\sigma^2\de 2$, which is equivalent to consider $X_u(t)\de\sqrt2(\sqrt2t-x_u(t))$ for $u\in\nt$ in our settings. And instead of looking at the highest particles, they study the extremal process seen from the lowest or leftmost particle $X(t)\de\min_{u\in\nt} X_u(t)=\sqrt 2(\sqrt 2t-x(t))$. \\

For $b>0$, define the process $\Gamma^{(b)}$ by

$$
\Gamma^{(b)}_s \de \left\{
\begin{array}{ll}
B_{s}, & \mbox{if } s \in [0,T_b], \\
b-R_{s-T_b}, & \mbox{if } s\geq T_b,
\end{array}
\right.
$$
\noindent
where $B$ is a standard Brownian motion, $T_b\de\inf\{t\geq 0 : B_t=b\}$ and $R$ is a three-dimensional Bessel process started from $0$ and independent from $B$. If $A$ is a measurable set of $C(\mathbb R_+,\mathbb R)$ and $b> 0$, we consider $Y$ whose law is given by
\begin{equation}\label{dens1}
\mathbb P(Y\in A,-\inf_{s\geq 0}Y(s)\in \mathrm db)=\frac{1}{c_1}\es\Big[\e^{-2\int_{0}^{\infty}G_v(\sigma \Gamma^{(b)}_v)\mathrm dv}\ind_{-\sigma\Gamma^{(b)}\in A}\Big],
\end{equation}
where $G_t$ is the distribution function of $X(t)$ and $c_1\de\int_{0}^{\infty}\es[\e^{-2\int_{0}^{\infty}G_v(\sigma \Gamma^{(b)}_v)\mathrm dv}]\mathrm db$, which is finite by \cite[Equation 6.7]{abbs2013}.
\begin{rem}
	If we denote $Y_t(s)\de X_{t}(t-s)-X(t)$ for $s\in [0,t]$, where $s\mapsto X_{t}(s)$ is the path followed by the leftmost particule at time $t$, the law of $Y$ is actually the limit of the law of $Y_t$ as $t\rightarrow \infty$, see \cite[Theorem 2.3]{abbs2013}, but we won't need this fact here.
\end{rem}

Conditionally on $Y$ distributed as in (\ref{dens1}), let $\pi$ be a Poisson point process on $[0, \infty)$ with intensity $2(1-G_t(-Y(t)) \mathrm dt$. For each point $t\in \pi$, start an
independent branching Brownian motion $(\mathscr N^*_{t} (s), s \geq 0)$ at position $Y(t)$ conditioned to $\min \mathscr N^*_{t} (t) > 0$ (here these BBM are considered as point processes). By \cite[Theorem 2.3]{abbs2013}, we have the following representation for the decoration process
\[
\mathcal Q \de \delta_0 + \sum_{t\in\pi} \mathscr N^*_{t} (t).
\]

And $\mathcal Q$ is related to the previous section's $\mathcal C$ by
\[
\mathcal C\sd\sum_{x\in \mathcal Q}\delta_{-\frac{1}{\sqrt 2}x}.
\]

We need to explicit a bit more the processes $(\mathscr N^*_{t},t \in\pi)$. Let $(X^y_u(s),s\geq 0, u\in\ns^y)_{y\geq 0}$ be a family of independent branching Brownian motion starting at $0$ with drift $2$ and variance $\sigma^2$. Define $X^y(s)\de\min_{u\in\ns^y}X^y_u(s)$ the position of the leftmost particle at time $s$. Then, conditionally on $Y$, introduce, for $t>0$, the processes $\mathscr N^*_{t}(s)\de \sum_{u\in\ns^t}\delta_{Y(t)+X^{t}_u(s)}$ where the BBM are conditioned to $X^t(t)+Y(t)>0$. The functional of Proposition \ref{sup} may now be expressed as
\begin{align*}
R_\beta &\de \sum_{x\in\mathcal C-\{0\}}\e^{\beta x}
\qquad~~= \sum_{x\in\mathcal Q-\{0\}}\e^{-\frac{\beta}{\sqrt 2} x}\\
&=\sum_{t\in\pi} \sum_{y\in \mathscr N^*_{t}(t)}\e^{-\frac{\beta}{\sqrt 2} y}
=\sum_{t\in\pi}\e^{-\frac{\beta}{\sqrt 2} Y(t)} \sum_{u\in\nt^t}\e^{-\frac{\beta}{\sqrt 2} X_u^t(t)}\\
&= \sum_{t\in\pi}\e^{-\frac{\beta}{\sqrt 2} Y(t)} C_t,
\end{align*}
where $C_t\de\sum_{u\in\nt^t}\e^{-\frac{\beta}{\sqrt 2} X_u^t(t)}$.\\
We want to prove that $0\in\mathrm{supp}(R_\beta)$. The two main ingredients are a rough estimation of the asymptotic of the paths of $Y$, see Appendix \ref{Bessel}, and a coupling with a simpler case were the BBMs arising in $C_t$ are not conditioned.\\

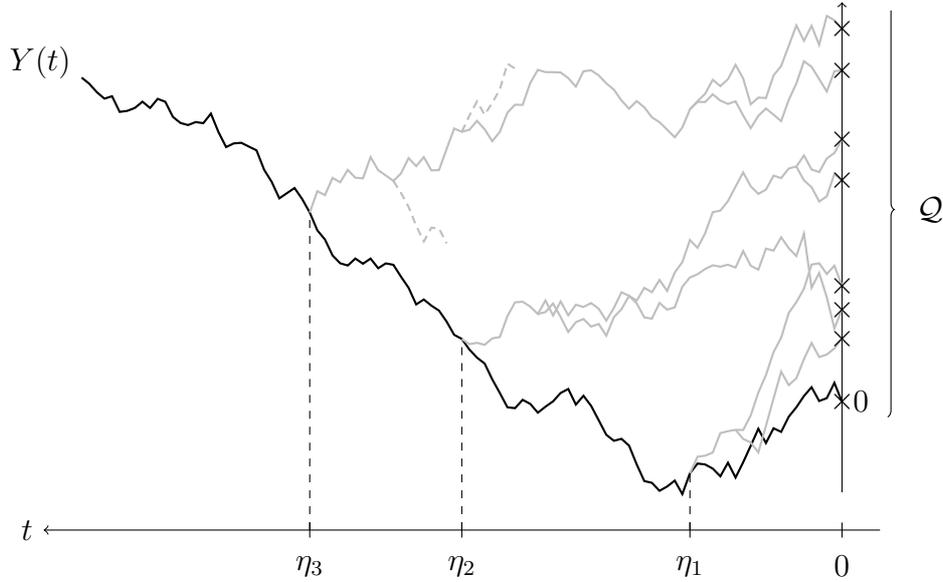
\begin{figure}
	\vspace{-0.5cm}
	\centering
	\subfigure{
		\begin{tikzpicture}[scale=1]
		
		\draw [->] (10.5,-6) -- (-.5,-6);
		\draw (-.5,-6) node[left] {$t$};
		\draw [->] (10,-5.5) -- (10,1);
		%\draw (10,1) node[above] {Decoration Process};
		\draw (0,0.2) node[left] {$Y(t)$};
		\draw (10,-6.1) -- (10,-5.9);
		\draw (10,-6.2) node[below]{$0$};
		
		\draw[color=black,decorate,decoration={brace,raise=0.1cm}]
		(10.5,1) -- (10.5,-4.5) node[right=0.2cm,pos=0.5] {~$\mathcal Q$};
		\draw[white,fill=white] (10.3,.9) rectangle (10.7,1.1);
		
		\draw[thick] (0,0)
		--(0.1,-0.0808541843931) 
		--(0.2,-0.193473930198) 
		--(0.3,-0.278945038333) 
		--(0.4,-0.247119924725) 
		--(0.5,-0.45500290165) 
		--(0.6,-0.44112084627) 
		--(0.7,-0.402244189101) 
		--(0.8,-0.31934902371) 
		--(0.9,-0.402741269425) 
		--(1.0,-0.28168094207) 
		--(1.1,-0.320452229789) 
		--(1.2,-0.520882280485) 
		--(1.3,-0.601873109199) 
		--(1.4,-0.629467082399) 
		--(1.5,-0.58936264852) 
		--(1.6,-0.604786086323) 
		--(1.7,-0.477321642146) 
		--(1.8,-0.723425369087) 
		--(1.9,-0.921546486533) 
		--(2.0,-0.872856510895) 
		--(2.1,-0.864821481412) 
		--(2.2,-0.914972706075) 
		--(2.3,-0.964258155794) 
		--(2.4,-1.22161911698) 
		--(2.5,-1.38340049515) 
		--(2.6,-1.60146085395) 
		--(2.7,-1.53279933369) 
		--(2.8,-1.46424467725) 
		--(2.9,-1.60998371871) 
		--(3.0,-1.78949753412) 
		--(3.1,-2.02560892433) 
		--(3.2,-2.14708001865) 
		--(3.3,-2.35818634405) 
		--(3.4,-2.4621698428) 
		--(3.5,-2.49197283358) 
		--(3.6,-2.3996726446) 
		--(3.7,-2.46852893455) 
		--(3.8,-2.40008391023) 
		--(3.9,-2.52811013868) 
		--(4.0,-2.463825551) 
		--(4.1,-2.48252892493) 
		--(4.2,-2.6400015923) 
		--(4.3,-2.78713571908) 
		--(4.4,-3.00766130202) 
		--(4.5,-2.94266937379) 
		--(4.6,-3.02414616374) 
		--(4.7,-3.08983347875) 
		--(4.8,-3.23508785175) 
		--(4.9,-3.41152227209) 
		--(5.0,-3.46555764733) 
		--(5.1,-3.60657108053) 
		--(5.2,-3.71394984482) 
		--(5.3,-3.79168051318) 
		--(5.4,-4.00101427413) 
		--(5.5,-4.18307479983) 
		--(5.6,-4.37262497869) 
		--(5.7,-4.38360143138) 
		--(5.8,-4.27279890997) 
		--(5.9,-4.32709707401) 
		--(6.0,-4.44400388867) 
		--(6.1,-4.36417244942) 
		--(6.2,-4.29306326602) 
		--(6.3,-4.18190670765) 
		--(6.4,-4.12921825822) 
		--(6.5,-4.33610254854) 
		--(6.6,-4.21911174174) 
		--(6.7,-4.35155911082) 
		--(6.8,-4.5426267794) 
		--(6.9,-4.73619517865) 
		--(7.0,-4.81622857759) 
		--(7.1,-4.7899890311) 
		--(7.2,-4.88081619316) 
		--(7.3,-5.12937363084) 
		--(7.4,-5.35011941549) 
		--(7.5,-5.37356948112) 
		--(7.6,-5.47937070678) 
		--(7.7,-5.42261516828) 
		--(7.8,-5.34429988014) 
		--(7.9,-5.52294575427)
		--(8.0,-5.24627386762) 
		--(8.1,-5.11760261973) 
		--(8.2,-5.13017766331) 
		--(8.3,-5.18491224122) 
		--(8.4,-5.28721957123) 
		--(8.5,-5.10710269986) 
		--(8.6,-5.30420332879) 
		--(8.7,-5.10239618866) 
		--(8.8,-4.88189896597) 
		--(8.9,-4.65707983702) 
		--(9.0,-4.84611563309) 
		--(9.1,-4.65010769476) 
		--(9.2,-4.69829520283) 
		--(9.3,-4.49689202026) 
		--(9.4,-4.37236476191) 
		--(9.5,-4.22235750588) 
		--(9.6,-4.10569421529) 
		--(9.7,-4.28885440601) 
		--(9.8,-4.2509541075) 
		--(9.9,-4.04783107518)
		--(10,-4.3)
		;
		\draw (10,-4.3) node{$\times$};
		\draw (10,-4.3) node[right]{$0$};
		%first branch
		\draw (8,-6.1) -- (8,-5.9);
		\draw (8,-6.2) node[below]{$\eta_1$};
		\draw[dashed] (8,-6) -- (8,-5.24627386762);
		\draw[thick,gray!50] (8.0,-5.24627386762)
		--(8.1,-5.0269937916) 
		--(8.2,-5.01416859434) 
		--(8.3,-4.97194062664) 
		--(8.4,-4.71420548244) 
		--(8.5,-4.68926572482) 
		--(8.6,-4.67146606207) 
		--(8.7,-4.6728437968) 
		--(8.8,-4.47923186699) 
		--(8.9,-4.17596581699) 
		--(9.0,-3.94453176435) 
		--(9.1,-3.63964534335) 
		--(9.2,-3.44936340652) 
		--(9.3,-3.12053940522) 
		--(9.4,-2.97766625558) 
		--(9.5,-2.79938694102) 
		--(9.6,-2.47957592626) 
		--(9.7,-2.46208520741) 
		--(9.8,-2.60454918308) 
		--(9.9,-2.49408314945) 
		--(10.0,-2.76003070867) 
		;
		\draw(10.0,-2.76003070867) node{$\times$};
		
		%\draw[thick] (9.4,-2.97766625558) 
		%--(9.5,-2.86058806193) 
		%--(9.6,-3.10798673476) 
		%--(9.7,-3.24056694578) 
		%--(9.8,-2.92064784903) 
		%--(9.9,-2.78303588787) 
		%--(10.0,-2.51705717874) 
		%;
		%\draw (10.0,-2.51705717874)  node{$\times$};
		
		\draw[thick,gray!50]  (8.6,-4.67146606207)
		--(8.7,-4.79122257849) 
		--(8.8,-4.83888013533) 
		--(8.9,-4.9701962889) 
		--(9.0,-4.61746650283) 
		--(9.1,-4.25481956644) 
		--(9.2,-3.89958277537) 
		--(9.3,-4.078699324) 
		--(9.4,-3.98658662688) 
		--(9.5,-3.73184464491) 
		--(9.6,-3.51876570298) 
		--(9.7,-3.7058122231) 
		--(9.8,-3.6730073148) 
		--(9.9,-3.59057404144) 
		--(10.0,-3.46105019637)
		;
		\draw (10.0,-3.46105019637)  node{$\times$};
		
		%second branch
		\draw (5,-6.1) -- (5,-5.9);
		\draw (5,-6.2) node[below]{$\eta_2$};
		\draw[dashed] (5,-6) -- (5.0,-3.46555764733);
		\draw[thick,gray!50] (5.0,-3.46555764733)
		--(5.1,-3.53474839698) 
		--(5.2,-3.51724943169) 
		--(5.3,-3.54397058781) 
		--(5.4,-3.45850844503) 
		--(5.5,-3.44019418912) 
		--(5.6,-3.16182957864) 
		--(5.7,-2.98074076284) 
		--(5.8,-2.97220525895) 
		--(5.9,-3.03742242542) 
		--(6.0,-3.13507442951) 
		--(6.1,-2.95553332895) 
		--(6.2,-3.15455441358) 
		--(6.3,-2.96896958842) 
		--(6.4,-3.15251343722) 
		--(6.5,-3.01187662587) 
		--(6.6,-3.04224356102) 
		--(6.7,-2.98951087068) 
		--(6.8,-3.06479715668) 
		--(6.9,-3.22271397614) 
		--(7.0,-3.05443872998) 
		--(7.1,-2.84652946005) 
		--(7.2,-2.89484224729) 
		--(7.3,-2.95781209877) 
		--(7.4,-2.78824784068) 
		--(7.5,-2.8462092963) 
		--(7.6,-2.68349704159) 
		--(7.7,-2.82911065192) 
		--(7.8,-2.58542920944) 
		--(7.9,-2.44670909623) 
		--(8.0,-2.15395596788) 
		--(8.1,-2.25690100073) 
		--(8.2,-2.06473266202) 
		--(8.3,-1.78444230701) 
		--(8.4,-1.76417835632) 
		--(8.5,-1.63833207802) 
		--(8.6,-1.44612169556) 
		--(8.7,-1.25690090251) 
		--(8.8,-1.44171429119) 
		--(8.9,-1.54261498876) 
		--(9.0,-1.48485207756) 
		--(9.1,-1.30022778637) 
		--(9.2,-1.33314137567) 
		--(9.3,-1.20534791419) 
		--(9.4,-1.17965833755) 
		--(9.5,-1.36259876921) 
		--(9.6,-1.45317190871) 
		--(9.7,-1.6435326857) 
		--(9.8,-1.57454739188) 
		--(9.9,-1.27847560983) 
		--(10.0,-1.36548619512) 
		;
		\draw (10.0,-1.36548619512)  node{$\times$};
		
		\draw[thick,gray!50] (6.0,-3.13507442951) 
		--(6.1,-3.1092818651259) 
		--(6.2,-2.99606925239) 
		--(6.3,-3.17573786876) 
		--(6.4,-3.37362451396) 
		--(6.5,-3.2429693206) 
		--(6.6,-3.20100422377) 
		--(6.7,-3.3019607257) 
		--(6.8,-3.26641974513) 
		--(6.9,-3.42258835623) 
		--(7.0,-3.17681249979) 
		--(7.1,-3.07075594777) 
		--(7.2,-2.89474535442) 
		--(7.3,-2.99377093013) 
		--(7.4,-3.17940096171) 
		--(7.5,-3.19720852667) 
		--(7.6,-2.98851756829) 
		--(7.7,-3.07733318097) 
		--(7.8,-3.0470973634) 
		--(7.9,-2.77269923076) 
		--(8.0,-2.65098315293) 
		--(8.1,-2.607856434) 
		--(8.2,-2.53796468181) 
		--(8.3,-2.48935809854) 
		--(8.4,-2.50740538843) 
		--(8.5,-2.48154664099) 
		--(8.6,-2.51948531344) 
		--(8.7,-2.29965813371) 
		--(8.8,-2.49465107407) 
		--(8.9,-2.54791787306) 
		--(9.0,-2.2680491527) 
		--(9.1,-2.29512398514) 
		--(9.2,-2.11662313584) 
		--(9.3,-2.26239631813) 
		--(9.4,-2.36622117595) 
		--(9.5,-2.06931809006) 
		--(9.6,-2.78599181365) 
		--(9.7,-2.59155372705) 
		--(9.8,-2.911587211216) 
		--(9.9,-3.32602463414) 
		--(10.0,-3.07300669407);
		\draw (10.0,-3.07300669407)  node{$\times$};
		
		\draw[thick,gray!50] (9.4,-1.17965833755)
		--(9.5,-1.032892835634) 
		--(9.6,-1.2187635468) 
		--(9.7,-1.25345309417) 
		--(9.8,-1.07297782405) 
		--(9.9,-0.999595884) 
		--(10.0,-0.819836862973)
		;
		\draw (10.0,-0.819836862973)  node{$\times$};
		
		%third branch
		\draw (3,-6.1) -- (3,-5.9);
		\draw (3,-6.2) node[below]{$\eta_3$};
		\draw[dashed] (3,-6) -- (3.0,-1.78949753412);
		\draw[thick, gray!50] (3.0,-1.78949753412)
		--(3.1,-1.5063149872) 
		--(3.2,-1.55861301237) 
		--(3.3,-1.32966827572) 
		--(3.4,-1.38840902161) 
		--(3.5,-1.22430982464) 
		--(3.6,-1.32807236979) 
		--(3.7,-1.33583302818) 
		--(3.8,-1.05459595913) 
		--(3.9,-1.22738887935) 
		--(4.0,-1.29770987434) 
		--(4.1,-1.37035134674) 
		--(4.2,-1.27143730388) 
		--(4.3,-1.06715643391) 
		--(4.4,-1.17465080199) 
		--(4.5,-1.05463469436) 
		--(4.6,-1.23043955629) 
		--(4.7,-1.22629489956) 
		--(4.8,-0.94701703745) 
		--(4.9,-0.674839918044) 
		--(5.0,-0.719387759293) 
		--(5.1,-0.691727820149) 
		--(5.2,-0.573948902935) 
		--(5.3,-0.685282576034) 
		--(5.4,-0.843409046666) 
		--(5.5,-0.719777328764) 
		--(5.6,-0.465035596283) 
		--(5.7,-0.359766369719) 
		--(5.8,-0.362043730999) 
		--(5.9,-0.0710365912344) 
		--(6.0,0.102851442889) 
		--(6.1,0.0726584312814) 
		--(6.2,0.0719356427931) 
		--(6.3,0.0676265622517) 
		--(6.4,-0.00691460988806) 
		--(6.5,-0.120005173345) 
		--(6.6,0.0839948530065) 
		--(6.7,0.0961980485335) 
		--(6.8,0.0299113578085) 
		--(6.9,-0.0762231132714) 
		--(7.0,-0.15940596577) 
		--(7.1,-0.304199892262) 
		--(7.2,-0.412680581281) 
		--(7.3,-0.311572199778) 
		--(7.4,-0.39394421978) 
		--(7.5,-0.525869190626) 
		--(7.6,-0.591895633381) 
		--(7.7,-0.689348701481) 
		--(7.8,-0.79205222928) 
		--(7.9,-0.66250140343) 
		--(8.0,-0.420520295382) 
		--(8.1,-0.322529223958) 
		--(8.2,-0.324206884016) 
		--(8.3,-0.221269089192) 
		--(8.4,-0.30483260536) 
		--(8.5,-0.349400778345) 
		--(8.6,-0.31421401229) 
		--(8.7,-0.502012985137) 
		--(8.8,-0.671048266023) 
		--(8.9,-0.506642253234) 
		--(9.0,-0.409244807345) 
		--(9.1,-0.492042960668) 
		--(9.2,-0.577286423468) 
		--(9.3,-0.295088678137) 
		--(9.4,-0.0776758459042) 
		--(9.5,0.216739107329) 
		--(9.6,0.143288811785) 
		--(9.7,-0.0201625410477) 
		--(9.8,-0.185439423781) 
		--(9.9,0.0832585486384) 
		--(10.0,0.0988595782389) 
		;
		\draw (10.0,0.0988595782389)  node{$\times$};
		
		\draw[thick,gray!50] (8.0,-0.420520295382) 
		--(8.1,-0.31235) 
		--(8.2,-0.214162057202) 
		--(8.3,-0.0299759269367) 
		--(8.4,-0.113066996694) 
		--(8.5,-0.0213301422542) 
		--(8.6,0.169415047792) 
		--(8.7,-0.011203079639) 
		--(8.8,-0.354318058699) 
		--(8.9,-0.265176531062) 
		--(9.0,-0.271413108179) 
		--(9.1,0.028319383751) 
		--(9.2,0.153689426101) 
		--(9.3,0.421699095099) 
		--(9.4,0.692096639539) 
		--(9.5,0.488110189202) 
		--(9.6,0.686117776413) 
		--(9.7,0.435030526078) 
		--(9.8,0.817949248739) 
		--(9.9,0.760934626066) 
		--(10.0,0.64639111389) 
		;
		\draw (10.0,0.64639111389)  node{$\times$};
		
		\draw[thick,gray!50,densely dashed] (5.0,-0.719387759293)
		--(5.1,-0.503758916643) 
		--(5.2,-0.316408333828) 
		--(5.3,-0.49227643479) 
		--(5.4,-0.335380368936) 
		--(5.5,-0.156951425399) 
		--(5.6,0.1850749924905) 
		--(5.7,0.1187059807501) 
		;
		
		\draw[thick,gray!50,densely dashed] (4.1,-1.37035134674)
		--(4.2,-1.52919382329) 
		--(4.3,-1.71483306891) 
		--(4.4,-1.96395258015) 
		--(4.5,-2.17310235674) 
		--(4.6,-1.99013355251) 
		--(4.7,-2.00165771174) 
		--(4.8,-2.19788477933)
		;
		\end{tikzpicture} 
	}
	\caption{The point process $\mathcal Q$ obtained with independent BBM issued from the backward path $Y$ at Poissonian times $(\eta_k)$ and conditioned to stay above $0$ at time $0$.}
	\label{fig}
\end{figure}
Let us begin with this simpler case: we omit the last conditioning $X^t(t)+Y(t)>0$ for the BBM, more precisely, let us define 
\begin{align*}
S_\beta \de \sum_{t\in\pi}\e^{-\frac{\beta}{\sqrt 2} Y(t)} D_t,
\end{align*}
where $D_t\de\sumt\e^{-\frac{\beta}{\sqrt 2} X_u^t(t)}$ and the BBM $(X^y_u(s),s\geq 0, u\in\ns^y)_{y\geq 0}$ are independent (and not conditioned as before) and also independent of $Y$ and $\pi$. The following lemma gives a sufficient condition for $0$ being in the support of $\sum_{t\in\pi}\e^{-\beta/\sqrt 2 Y(t)} D_t$. Its proof is postponed to Appendix \ref{tech}.

\begin{lem}\label{support1}
	Let $\mu$ be a Radon measure on $\re_+$ and $\mathcal P$ a $\mathrm{PPP}(\mu)$. Let $(A_t)_{t \in \mathbb R_+}$ be independent positive random variables, independent of $\mathcal P$, and $f$ a positive and measurable function on $\mathbb R_+$, then
	$$\int_{0}^{\infty}\es\big[(f(t)A_t)\wedge 1\big] \mu(\mathrm d t)<\infty \Rightarrow 0 \in \mathrm{supp}\Big(\sum_{t\in\mathcal P} f(t)A_{t}\Big).$$
\end{lem}
\noindent
The idea is to apply Lemma \ref{support1} conditionally on a given path $Y$ with $A_t = D_t$, $f_Y(t)\de\e^{-\frac{\beta}{\sqrt 2} Y(t)}$ and  $\mu_Y(\mathrm dt)\de 2(1-G_t(-Y(t)) \mathrm dt$: if conditionally on almost every path $Y$, $0$ is in the support of $S_\beta$, it is therefore in the support of the unconditional law.
Let us prove that the following expectation is integrable w.r.t. $\mathrm dt$ on $\mathbb R_+$ 
\begin{align}\label{integrable}
\es\big[(f_Y(t)D_t)\wedge 1\,|\,Y\big]&=\mathbb P\Big(D_t>\frac{1}{f_Y(t)}\,|\,Y\Big)+f_Y(t)\,\es\big[D_t\ind_{D_t\leq 1/f_Y(t)}\,|\,Y\big].
\end{align}
It will be more convenient for the computations to come to use the normalization of the previous section with $X^y_u(t)=\sqrt2(\sqrt2t-x^y_u(t))$. This way, we have $D_t=\sumt \e^{\beta(x^t_u(t)-\sqrt 2t)}$ and the superadditivity of the function $x\mapsto x^{\beta/\sqrt 2}$ gives
\[
D_t=\sumt \e^{\beta(x^t_u(t)-\sqrt 2t)} \leq \Big(\sumt \e^{\sqrt2(x_u^t(t)-\sqrt 2t)}\Big)^{\frac{\beta}{\sqrt 2}},
\]
then
\begin{align*}
\mathbb P(D_t>\frac{1}{f_Y(t)}\,|\,Y)
&\leq \mathbb P\Bigg(\Big(\sumt \e^{\sqrt2(x_u^t(t)-\sqrt 2t)}\Big)^{\frac{\beta}{\sqrt 2}}>\frac{1}{f_Y(t)}\,|\,Y\Bigg)\\
&=\mathbb P\Big(\sumt \e^{\sqrt2(x_u^t(t)-\sqrt 2t)}>\e^{Y(t)}\,|\,Y\Big)\\
&\leq \e^{-Y(t)},
\end{align*}
the last inequality being a consequence of Markov's inequality and $\es[\sumt \e^{\sqrt2(x_u^t(t)-\sqrt 2t)}\,|\,Y]=1$. The last term is integrable almost surely by Proposition \ref{Yest}.\\
For the second term in (\ref{integrable}), observe that $D_t\leq 1/f_Y(t)$ implies $\e^{\beta(x_u^t(t)-\sqrt 2t)}\leq \e^{\frac{\beta}{\sqrt 2} Y(t)}$ for every $u\in\nt^t$ and the many-to-one lemma, see \cite[Theorem 1.1]{shi2015} for example, yields
\begin{align*}
\es[D_t\ind_{D_t\leq 1/f_Y(t)}\,|\,Y]&\leq \es\Big[\sumt \e^{\beta(x_u^t(t)-\sqrt 2t)}\ind_{x_u^t(t)\leq\sqrt2 t+\frac{1}{\sqrt 2}Y(t)}\,|\,Y\Big]\\
& = \e^t\,\es\left[\e^{\beta(\sqrt t G-\sqrt 2t)}\ind_{\sqrt t G\leq\sqrt2 t+\frac{1}{\sqrt 2}Y(t)}\,|\,Y\right]\\
& = \e^{t-\sqrt 2\beta t}\,\es\left[\e^{\beta\sqrt{t} G}\ind_{G\leq\sqrt{2t}+\frac{1}{\sqrt {2t}}Y(t)}\,|\,Y\right],
\end{align*}
where $G\sim\norm$ is independent from $Y$. The last expectation can be handled with the Gaussian estimate
\[
\es[\e^{\lambda G}\ind_{G\leq x}]\leq \e^{\lambda x-\frac{x^2}{2}}, \qquad \forall \, \lambda>x>0,
\]
whose proof can be found in Appendix \ref{tech}. Proposition \ref{Yest} ensures that almost surely there exists $ t_0= t_0(Y)$ such that for $t\geq t_0$, $\beta\sqrt t>\sqrt{2t}+\frac{1}{\sqrt{2t}}Y(t)$, which implies
\[
\es[D_t\ind_{D_t\leq 1/f_Y(t)}\,|\,Y]  \leq \e^{t-\sqrt 2\beta t} \e^{\beta \sqrt{t}(\sqrt{2t}+\frac{1}{\sqrt{2t} }Y(t))-\frac{1}{2}(\sqrt{2t}+\frac{1}{\sqrt {2t}}Y(t))^2} \leq \e^{(\frac{\beta}{\sqrt 2}-1) Y(t)},
\]
and thus, for $t$ large enough,
\begin{align*}
f_Y(t)\es\big[D_t\ind_{D_t\leq 1/f_Y(t)}\,|\,Y\big]
&\leq \e^{-Y(t)}.
\end{align*}
We finally conclude that the two terms in \eqref{integrable} are almost surely integrable w.r.t to $\mathrm dt$ on $\mathbb R_+$ thanks to Proposition \ref{Yest}. Therefore, recalling that $\mu_Y(\dt)=2(1-G_t(-Y(t)) \mathrm dt$ on $\mathbb R_+$, we have

\begin{equation}\label{simpl}
\int_{0}^{\infty}\es[(f_Y(t)D_t)\wedge 1\,|\,Y]\,\mu_Y(\mathrm dt)\leq
\int_{0}^{\infty}\es[(f_Y(t)D_t)\wedge 1\,|\,Y]\,2\mathrm dt<\infty \quad\text{a.s.}
\end{equation}
and Lemma \ref{support1} shows that $0$ is in the support of $S_\beta$.\\

Let us go back to the original case with $R_\beta$. We now construct a coupling between $R_\beta$ and $S_\beta$. For $t>0$, denote $D^{(1)}_t\de D_t$ and $X_t^{(1)}\de X^t(t)$, and let $((D_t^{(i)},X_t^{(i)}),\, i\geq 1)$ be an i.i.d. sequence. We then redefine $C_t$ as the first $D_t^{(i)}$ such that $X_t^{(i)}+Y(t)>0$. The idea of the following is to use the fact that $C_t$ and $D_t$ coincides with large probability for $t$ large enough in $\pi$.
Let $\eta_1\leq\eta_2\leq...$ denote the atoms of $\pi$ ranked in non-decreasing order and define $T\de\sup\{\eta_i : C_{\eta_i}\neq D_{\eta_i}\}$, we have
\begin{eqnarray*}
&&\int_{0}^{\infty}\es[(f_Y(t)C_t)\wedge 1\,|\,Y] \,\mu_Y(\dt)
\\
&\leq& \int_{0}^{\infty}\es[(f_Y(t)C_t\ind_{t<T})\wedge 1\,|\,Y]\, 2\dt+\int_{0}^{\infty}\es[(f_Y(t)D_t\ind_{t\geq T})\wedge 1\,|\,Y]\, 2\dt
\\
&\leq& 2\int_{0}^{\infty}\p(T>t\,|\,Y) \,\dt+\int_{0}^{\infty}\es[(f_Y(t)D_t)\wedge 1\,|\,Y]\, 2\dt,
\end{eqnarray*}
where the second integral is almost surely finite thanks to Equation (\ref{simpl}). We end this section by proving that the first one is also finite almost surely.

\begin{lem}
$\int_{0}^{\infty}\p(T>t\,|\,Y) \, \dt$ is finite almost surely.
\end{lem}
\begin{proof}
	$C_t\neq D_t$ occurs when $Y(t)+X^t(t)\leq0$, we thus need to control the probability of such an event. Using $\txtt\de\max_{u\in\nt^t}x^t_u(t)-m_t$, one gets
	\begin{align*}
	\p(Y(t)+X^t(t)\leq 0\,|\,Y)&=\p\big(\txtt\geq\frac{1}{\sqrt 2}Y(t)+\frac{3}{2\sqrt 2}\log(t)\,|\,Y\big)
	\leq \p\big(\txtt\geq\frac{1}{\sqrt 2}Y(t)\wedge \sqrt t\,|\,Y\big),
	\end{align*}
	where $\sqrt t$ is here to fulfill the condition $t\geq A^2$ of Proposition \ref{Bram}. Choose $t_0$ large enough, depending on $Y$, such that $Y(t)\geq0$ for $t\geq t_0$. Then Proposition \ref{Bram} implies
	\[
	\p(Y(t)+X^t(t)\leq0\,|\,Y) \leq c\,\Big(\frac{1}{\sqrt 2}Y(t)\wedge \sqrt t+1\Big)^2\e^{-Y(t)\wedge \sqrt{2t}},
	\qquad \text{for } t\geq t_0.
	\]
	We thus have, for $c'>0$ large enough,
	\[
	\p(C_t\neq D_t\,|\,Y)\leq c' t \e^{-Y(t)\wedge\sqrt{2t}}, \qquad \text{for } t\geq t_0.
	\]
	Applying the union bound, with $t>t_0$, gives
	\begin{align*}
	\p(T>t\,|\,Y,\pi)&= \p(\exists i\geq 1 : \eta_i>t ,C_{\eta_i}\neq D_{\eta_i}\,|\,Y,\pi)
	\leq
	\sum_{i:\eta_i>t} \p(C_{\eta_i}\neq D_{\eta_i}\,|\,Y,\pi)
	\\
	&\leq c'\sum_{i:\eta_i>t}  \eta_i \e^{-Y(\eta_i)\wedge\sqrt{2\eta_i}}.
	\end{align*}
Taking expectation with respect to $\pi$ together with Campbell's theorem, see \cite[Section 3.2]{Kingman93} for example, yields
$$
	\p(T>t\,|\,Y)\leq c'\int_{t}^{\infty}  s \e^{-Y(s)\wedge\sqrt{2s}} \mu_Y(\mathrm ds)
	\leq c'\int_{t}^{\infty} s \e^{-Y(s)\wedge\sqrt{2s}} \,2 \, \ds,
$$
and finally
$$
	\int_{t_0}^{\infty} \p(T>t\,|\,Y)\,\dt
	\leq c'\int_{t_0}^{\infty}\left(\int_{t}^{\infty} s\e^{-Y(s)\wedge\sqrt{2s}} \,2\ds\right)\dt
	= 2c'\int_{t_0}^{\infty} s(s-t_0)\e^{-Y(s)\wedge\sqrt{2s}} \,\ds,
$$
	which is almost surely finite thanks to Proposition \ref{Yest}.
\end{proof}
Finally, Lemma \ref{support1} shows that $0\in\mathrm{supp}(R_\beta)$ and Theorem \ref{thm2} is proved.
%%%%%%%%%%%%%%%%%%%%%%%%%%%%%%%%%%%%%%%%%%%%%%%%%%%%%%%%%

\addcontentsline{toc}{section}{Appendix}
\appendix
\begin{appendix}

%%%%%%%%%%%%%%%%%%%%%%%%%%%%%%%%%%%%%%
\section{Appendix}
%%%%%%%%%%%%%%%%%%%%%%%%%%%%%%%%%%%%%%

%%%%%%%%%%%%%%%%%%%
\subsection{$\rho_{\beta,t}$ and extremal particles}
\label{a}
%%%%%%%%%%%%%%%%%%%

Let us denote $N_{t}^A$ the number of particles whose centered positions are above level $-A$ at time $t$: 
$$N^A_t\de\mathcal{E}_t([-A,+\infty[)=|\{u\in\nt : \txtu\geq -A \}|, \qquad \text{for } A,t\geq0.$$
We will need the following consequence of \cite[Lemma 4.2]{cortineshartunglouidor2019}, where the authors obtained a detailed description of the extreme level sets of the BBM.

\begin{prop}[Cortines, Hartung and Louidor \cite{cortineshartunglouidor2019}]\label{nat}
	There exists $C>0$ such that, for all $A\geq0$, 
	$\es[N_t^A\,;\,\txt \leq A] \leq C(A+1)^2 \e^{\sqrt2 A}.$
\end{prop}
We can now prove Proposition \ref{pro1}.
\begin{prop} Let $\eta>0$, then $\underset{A\rightarrow \infty}{\mathrm{lim}} \underset{t\rightarrow \infty}{\mathrm{limsup~}} \mathbb P(\rho_{\beta,t}^{]-\infty,-A]}(\mathbb R_+)>\eta )=0.$
\end{prop}
\begin{proof}
Fix $\delta>0$ such that $\sqrt2+\delta<\beta$ and define the event $E_{A,t}\coloneqq\{\exists n\geq 0 : N_t^{A+n}\geq \e^{(\sqrt2+\delta)(A+n)}\}$ for $A,t\geq0$. We can break down the probability in the proposition the following way
\begin{align}\label{majo}
\mathbb P(\rho_{\beta,t}^{]-\infty,-A]}(\mathbb R_+) >\eta)& \leq \mathbb P(\rho_{\beta,t}^{]-\infty,-A]}(\mathbb R_+) >\eta,\, \txt\leq A)+\mathbb P(\txt> A)\notag \\ 
& \leq \mathbb P(E_{A,t},\,\txt \leq A)+\mathbb P(\rho_{\beta,t}^{]-\infty,-A]}(\mathbb R_+) >\eta,\,E_{A,t}^c)+c(A+1)^2 \e^{-\sqrt 2 A},
\end{align}
whenever $t\geq A^2$ thanks to Proposition \ref{Bram}.\\
Let us begin with the first term in (\ref{majo}). Proposition \ref{nat} and Markov's inequality yield
\begin{equation}
\mathbb P(N_t^A\geq \e^{(\sqrt2+\delta)A},\,\txt \leq A) \leq C(A+1)^2 \e^{-\delta A},
\label{markov}
\end{equation}
such that
\begin{align}\label{Eat}
\mathbb P(E_{A,t},\,\txt \leq A) & \leq \sum_{n\geq 0}
\mathbb P(N_t^{A+n}\geq \e^{(\sqrt2+\delta)(A+n)},\,\txt \leq A)\notag\\
& \leq \sum_{n\geq 0}
\mathbb P(N_t^{A+n}\geq \e^{(\sqrt2+\delta)(A+n)},\,\txt \leq A+n)\notag\\
& \leq C\e^{-\delta A} \sum_{n\geq 0} (A+n+1)^2 \e^{-\delta n},
\end{align}
where the last term does not depend on $t$ and tends to $0$ when $A\rightarrow\infty$.\\
For the second term in (\ref{majo}), observe that on the event $E_{A,t}^c$, one has
\begin{align*}
	\rho_{\beta,t}^{]-\infty,-A]}(\mathbb R_+) & \leq \sum_{n\geq 0} N_t^{A+n+1} \e^{-\beta(A+n)}
	 \leq \sum_{n \geq 0} \e^{(\sqrt2+\delta)(A+n+1)} \e^{-\beta(A+n)}\\
	& \leq \frac{\e^{\sqrt 2 +\delta}}{1-\e^{\sqrt 2 +\delta-\beta}}\e^{(\sqrt 2+\delta-\beta)A} \underset{A\rightarrow \infty}{\longrightarrow} 0.
\end{align*}
\noindent
Therefore, for $A$ large enough, and independent of $t$, $\mathbb P(\rho_{\beta,t}^{]-\infty,-A]}(\mathbb R_+) >\eta,\,E_{A,t}^c)$ vanishes. Combining this fact with (\ref{Eat}) in the previous inequality (\ref{majo}) concludes the proof.
\end{proof}
The following lemma shows that the Gibbs measure under $\beta>\beta_c$ is concentrated on the extremal particles.
\begin{lem} Let $\eta>0$, then
	$\lim\limits_{A\to \infty} \limsup\limits_{t\to\infty} \p(\mathcal G_{\beta,t}(\nt([-A,A]^c))>\eta)=0$.
\end{lem}
\begin{proof}
	We can express $\mathcal G_{\beta,t}$ with $\rho_{\beta,t}$ as follows
	\[
	\mathcal G_{\beta,t}(\nt([-A,A]^c))=\frac{\rho^{[-A,A]^c}_{\beta,t}(\mathbb R_+)}{\rho_{\beta,t}(\mathbb R_+)}.
	\]
	Recall, from Proposition \ref{mass}, that 
	$\rho_{\beta,t}(\mathbb{R}_+) {\xrightarrow[]{\mathrm{(d)}}} \rho_{\beta}(\mathbb R_+),$ when $t \to \infty$,
	where the limit is a positive random variable.
	If $\varepsilon>0$, one can therefore find $\delta>0$ and $t_0>0$ such that, for all $t>t_0$,
	$\p(\rho_{\beta,t}(\mathbb R_+)>\delta)>1-\varepsilon.$
	Then
$
\p(\mathcal G_{\beta,t}(\nt([-A,A]^c))>\eta)\leq \varepsilon + \p(\rho^{[-A,A]^c}_{\beta,t}(\mathbb R_+)>\eta\delta).
$
Taking the limsup in $t$ and then the limit in $A$ using \eqref{rho} concludes the proof of the lemma.
\end{proof}	

%%%%%%%%%%%%%%%%%%%
\subsection{Asymptotics of $Y$ and Bessel Processes}
\label{Bessel}
%%%%%%%%%%%%%%%%%%%
	
The aim of this section is to prove the following proposition.
\begin{prop}\label{Yest}
Let $Y$ be the backward path introduced in Section \ref{lapmet} and $\varepsilon>0$, then  
\[
\p\big(\exists \,  t_0\in\re_+: t^{\frac{1}{2}-\varepsilon}<Y(t)<t^{\frac{1}{2}+\varepsilon} \text{ for all } t\geq t_0\big)=1.
\]
\end{prop}
\begin{proof}
We begin by proving the aformentioned statement for a three-dimensional Bessel process $R$ started at $0$, represented as $R_t\de |B_t|$, $B$ being a standard Brownian motion in $\mathbb R^3$. Let $f$ be an increasing function on $(0,\infty)$, the Dvoretzky-Erd\"os test, see \cite[Theorem 3.22]{morters_peres_2010}, states that
\[
\int_{1}^{\infty}f(t)t^{-\frac{3}{2}}\dt<\infty \qquad \text{if and only if} \qquad
\liminf\limits_{t \to \infty} \frac{|B_t|}{f(t)}=\infty  \quad \text{a.s.}
\]
Considering the function $f(t)\de t^{1/2-\varepsilon}$ yields 
\[\liminf\limits_{t \to \infty} \frac{|B_t|}{t^{\frac{1}{2}-\varepsilon}}=\infty\quad \text{a.s.},\]
thus $\p(\exists t_0>0 : t^{1/2-\varepsilon}<R_t\text{ for } t\geq t_0)=1$.
The law of the iterated logarithm applied to the components of $B$, see \cite[Theorem 8.5.1]{durrett2019} for instance, gives
\[
\limsup\limits_{t \to \infty}\frac{|B_t|}{t^{\frac{1}{2}+\varepsilon}}= 0\quad \text{a.s.},
\]
and therefore
\begin{equation}\label{bessel_estimate}
\p\big(\exists t_0\in\re_+: t^{\frac{1}{2}-\varepsilon}<R_t<t^{\frac{1}{2}+\varepsilon} \text{ for } t\geq t_0\big)=1.
\end{equation}
Now recall that $\Gamma^{(b)}$ is defined as follows, for $b>0$,
$$
\Gamma^{(b)}_s = \left\{
\begin{array}{ll}
B_{s}, & \mbox{if } s \in [0,T_b], \\
b-R_{s-T_b}, & \mbox{if } s\geq T_b,
\end{array}
\right.
$$
where $T_b\de\inf\{t\geq 0 : B_t=b\}$ and $R$ is a three-dimensional Bessel process started at $0$ and independent from $B$. Moreover 
\begin{equation*}\label{dens}
\mathbb P(Y\in A,-\inf_{s\geq 0}Y(s)\in \mathrm db)=\frac{1}{c_1}\es\big[\e^{-2\int_{0}^{\infty}G_v(\sigma \Gamma^{(b)}_v)\mathrm dv}\ind_{-\sigma\Gamma^{(b)}\in A}\big],
\end{equation*}
for any measurable set $A$ of $ C(\re_+,\re)$.
Consequently, conditionally on $-\inf_{s\geq 0}Y(s)=b$, the law of $Y$ is absolutely continuous with respect to the law of $-\sigma \Gamma^{(b)}$. And using the fact that $\Gamma^{(b)}(s)=b-R_{s-T_b}$ for $s\geq T_b$, it is not hard to see that if we let $A=\{ f\in C(\re_+,\re) : \exists t_0\in\re_+ ~ \forall t\geq t_0, ~ t^{1/2-\varepsilon}<f(t)< t^{1/2+\varepsilon} \}$, we have $\p(-\sigma\Gamma^{(b)}\in A)=1$ for every $b>0$, thanks to (\ref{bessel_estimate}). We finally obtain
\[
\p(Y\in A)=\int_{0}^{\infty}\p(Y\in A\,|-\inf_{s\geq 0}Y(s)\in \mathrm db)\p(-\inf_{s\geq 0}Y(s)\in \mathrm db)=1.
\]
\end{proof}
%%%%%%%%%%%%%%%%%%%
\subsection{Auxiliary results}\label{tech}
%%%%%%%%%%%%%%%%%%%

Let $X$ be a non-negative random variable and define the log-Laplace transform of $X$, which is the real number $\phi(\lambda)$ verifying
$$\e ^{-\phi(\lambda)}\underset{}{=}\es\big[\e^{-\lambda X}\big], \qquad \forall\lambda\geq0.$$
The following lemma characterizes the fact that $\mathbb P(X<\varepsilon)>0$ for all $\varepsilon>0$, i.e. that $0$ belongs to the support of the law of $X$ denoted supp$(X)$.
\begin{lem}\label{lapl} If $X$ is a non-negative random variable then
	$0 \in \mathrm{supp}(X) \Leftrightarrow \phi(\lambda) =o(\lambda) \text{ when }\lambda \rightarrow \infty.$
\end{lem}
\begin{proof}
	Assume that $0\in\mathrm{supp}(X)$, then for every $\varepsilon>0$, $\es[\e^{-\lambda X}]\geq \mathbb P(X< \varepsilon)\e^{-\varepsilon \lambda}>0$, therefore $\phi(\lambda)\leq\lambda\varepsilon-\log\p(X<\varepsilon)$. And if $0\notin\mathrm{supp}(X)$, then pick $\varepsilon>0$ such that $X>\varepsilon$ a.s. so that $\es[\e^{-\lambda X}]\leq \e^{-\lambda\varepsilon}$ and $\phi(\lambda)\geq\lambda\varepsilon$.
\end{proof}

\begin{lem}\label{support}
	Let $\mu$ be a Radon measure on $\re_+$ and $\mathcal P$ a $\mathrm{PPP}(\mu)$. Let $(A_t)_{t \in \mathbb R_+}$ be independent positive random variables, independent of $\mathcal P$, and $f$ a positive and  measurable function on $\mathbb R_+$, then
$$\int_{0}^{\infty}\es\big[(f(t)A_t)\wedge 1\big] \mu(\mathrm d t)<\infty \,\Rightarrow \, 0 \in \mathrm{supp}\Big(\sum_{t\in\mathcal P} f(t)A_{t}\Big).$$
\end{lem}

\begin{proof}
	Taking expectation with respect to $(A_t)$ first and using Campbell's theorem, see \cite[Section 3.2]{Kingman93}, leads to the Laplace transform 
	$$\es\left[\e^{-\lambda \sum_{t\in\mathcal P} f(t)A_{t}}\right] = \e^{-\int_{0}^{\infty}\es[1-\e^{-\lambda f(t)A_t}]\mu(\mathrm dt)}.$$
	Let $\phi(\lambda)$ denote the log-Laplace transform of $\sum_{t\in\mathcal P} f(t)A_{t}$ so that
	\begin{align*}
	\frac{1}{\lambda}\phi(\lambda)=\int_{0}^{\infty}\es\left[\frac{1-\e^{-\lambda f(t)A_t}}{\lambda}\right] ~\mu(\mathrm dt).
	\end{align*}
	The dominated convergence theorem used twice together with the inequality $$\es\left[\frac{1-\e^{-\lambda f(t)A_t}}{\lambda}\right]\leq \es\big[(f(t)A_t)\wedge 1\big], \qquad \forall\lambda\geq 1,$$ show that $\phi(\lambda)/\lambda$ tends to $0$ when $\lambda \rightarrow \infty$ and Lemma \ref{lapl} concludes the proof.
\end{proof}

\begin{lem}
	Let $G\sim\mathcal N(0,1)$ and $\lambda>x>0$, then $\es[\e^{\lambda G}\ind_{G\leq x}]\leq \e^{\lambda x-\frac{x^2}{2}}$.
\end{lem}

\begin{proof}
Observe that
$$
\es\big[\e^{\lambda G}\ind_{G\leq x}\big]=\int_{-\infty}^{x} \e^{\lambda u}\e^{-\frac{u^2}{2}} \frac{\mathrm du}{\sqrt{2\pi}}
	=\e^{\lambda^2/2}\int_{-\infty}^{x} \e^{-\frac{(u-\lambda)^2}{2}} \frac{\mathrm du}{\sqrt{2\pi}}
	=\e^{\lambda^2/2}\int_{\lambda -x}^{\infty} \e^{-\frac{t^2}{2}} \frac{\mathrm dt}{\sqrt{2\pi}},
$$	
and the last integral is simply $\mathbb P(G>\lambda -x)$ which is bounded by $\e^{-\frac{(\lambda-x)^2}{2}}$.
\end{proof}
\noindent The following lemma is called the Gaussian integration by parts, see \cite[Equation (A.17)]{talagrand2011} for a proof.
\begin{lem}
	\label{IPP} 
	Let $X=(X_i)_{i\in I}$ be a centered Gaussian vector where $I$ is finite. Then, for any $C^1$ function $F \colon \re^I \rightarrow \re$ of moderate growth at infinity, we have for every $i\in I$
	\begin{equation*}
	\es \left[ X_i F(X) \right] 
	= \sum_{j\in I}  \es \left[ X_i X_j \right] \es \left[\partial_j F(X) \right].
	\end{equation*}
\end{lem}

\end{appendix}

%%%%%%%%%%%%%%%%%%%%%%%%%%%%%%%%%%%%%%%%%%%%%%%%%%%%%%%%%

\subsection*{Acknowledgements} I wish to thank my supervisor Olivier Zindy for introducing me to this subject and for useful discussions. I would like to also thank Zhan Shi for explaining precisely to me the description of the distribution  of the decoration process obtained in \cite{abbs2013}.

%%%%%%%%%%%%%%%%%%%%%%%%%%%%%%%%%%%%%%%%%%%%%%%%%%%%%%%%%

\addcontentsline{toc}{section}{References}

\bibliographystyle{abbrv}
\bibliography{biblio}

\def\cprime{$'$}
\begin{thebibliography}{10}

\bibitem{abbs2013}
E.~A{\"{\i}}d{\'e}kon, J.~Berestycki, {\'E}.~Brunet, and Z.~Shi.
\newblock Branching {B}rownian motion seen from its tip.
\newblock {\em Probab. Theory Related Fields}, 157(1-2):405--451, 2013.

\bibitem{arguin_2016}
L.-P. Arguin.
\newblock {\em Extrema of Log-correlated Random Variables: Principles and
  Examples}, pages 166--204.
\newblock Cambridge University Press, 2016.

\bibitem{abk2011}
L.-P. Arguin, A.~Bovier, and N.~Kistler.
\newblock Genealogy of extremal particles of branching {B}rownian motion.
\newblock {\em Comm. Pure Appl. Math.}, 64(12):1647--1676, 2011.

\bibitem{abk2013}
L.-P. Arguin, A.~Bovier, and N.~Kistler.
\newblock The extremal process of branching {B}rownian motion.
\newblock {\em Probab. Theory Related Fields}, 157(3-4):535--574, 2013.

\bibitem{arguinzindy2014}
L.-P. Arguin and O.~Zindy.
\newblock Poisson-{D}irichlet statistics for the extremes of a log-correlated
  {G}aussian field.
\newblock {\em Ann. Appl. Probab.}, 24(4):1446--1481, 2014.

\bibitem{arguinzindy2015}
L.-P. Arguin and O.~Zindy.
\newblock Poisson-{D}irichlet statistics for the extremes of the
  two-dimensional discrete {G}aussian free field.
\newblock {\em Electron. J. Probab.}, 20:no. 59, 19, 2015.

\bibitem{biskuplouidor2018}
M.~Biskup and O.~Louidor.
\newblock Full extremal process, cluster law and freezing for the
  two-dimensional discrete {G}aussian {F}ree {F}ield.
\newblock {\em Adv. Math.}, 330:589--687, 2018.

\bibitem{bovierhartung2017}
A.~Bovier and L.~Hartung.
\newblock Extended convergence of the extremal process of branching {B}rownian
  motion.
\newblock {\em Ann. Appl. Probab.}, 27(3):1756--1777, 2017.

\bibitem{bovierkurkova2004-2}
A.~Bovier and I.~Kurkova.
\newblock Derrida's generalized random energy models. {II}. {M}odels with
  continuous hierarchies.
\newblock {\em Ann. Inst. H. Poincar\'e Probab. Statist.}, 40(4):481--495,
  2004.

\bibitem{bramson78}
M.~D. Bramson.
\newblock Maximal displacement of branching {B}rownian motion.
\newblock {\em Comm. Pure Appl. Math.}, 31(5):531--581, 1978.

\bibitem{cortineshartunglouidor2019}
A.~Cortines, L.~Hartung, and O.~Louidor.
\newblock {The structure of extreme level sets in branching Brownian motion}.
\newblock {\em Ann. Probab.}, 47(4):2257 -- 2302, 2019.

\bibitem{Derrida1981}
B.~Derrida.
\newblock Random-energy model: an exactly solvable model of disordered systems.
\newblock {\em Phys. Rev. B (3)}, 24(5):2613--2626, 1981.

\bibitem{DerridaMottishaw_2021}
B.~Derrida and P.~Mottishaw.
\newblock One step replica symmetry breaking and overlaps between two
  temperatures.
\newblock {\em Journal of Physics A: Mathematical and Theoretical},
  54(4):045002, jan 2021.

\bibitem{derridaspohn88}
B.~Derrida and H.~Spohn.
\newblock Polymers on disordered trees, spin glasses, and traveling waves.
\newblock {\em J. Statist. Phys.}, 51(5-6):817--840, 1988.
\newblock New directions in statistical mechanics (Santa Barbara, CA, 1987).

\bibitem{durrett2019}
R.~Durrett.
\newblock {\em Probability: Theory and Examples}.
\newblock Cambridge Series in Statistical and Probabilistic Mathematics.
  Cambridge University Press, fifth edition, 2019.

\bibitem{ikedanagasawawatanabe68a}
N.~Ikeda, M.~Nagasawa, and S.~Watanabe.
\newblock {Branching Markov processes I}.
\newblock {\em Journal of Mathematics of Kyoto University}, 8(2):233 -- 278,
  1968.

\bibitem{ikedanagasawawatanabe68b}
N.~Ikeda, M.~Nagasawa, and S.~Watanabe.
\newblock {Branching Markov processes II}.
\newblock {\em Journal of Mathematics of Kyoto University}, 8(3):365 -- 410,
  1968.

\bibitem{ikedanagasawawatanabe69}
N.~Ikeda, M.~Nagasawa, and S.~Watanabe.
\newblock {Branching Markov Processes III}.
\newblock {\em Journal of Mathematics of Kyoto University}, 9(1):95 -- 160,
  1969.

\bibitem{jagannath2016}
A.~Jagannath.
\newblock On the overlap distribution of branching random walks.
\newblock {\em Electron. J. Probab.}, 21:no. 50, 16, 2016.

\bibitem{kallenberg2017}
O.~Kallenberg.
\newblock {\em Random measures, theory and applications}, volume~77 of {\em
  Probability Theory and Stochastic Modelling}.
\newblock Springer, Cham, 2017.

\bibitem{Kingman93}
J.~F.~C. Kingman.
\newblock {\em Poisson processes}, volume~3 of {\em Oxford Studies in
  Probability}.
\newblock The Clarendon Press Oxford University Press, New York, 1993.
\newblock Oxford Science Publications.

\bibitem{kurkova2003}
I.~Kurkova.
\newblock Temperature dependence of the {G}ibbs state in the random energy
  model.
\newblock {\em J. Statist. Phys.}, 111(1-2):35--56, 2003.

\bibitem{lalleysellke87}
S.~P. Lalley and T.~Sellke.
\newblock A conditional limit theorem for the frontier of a branching
  {B}rownian motion.
\newblock {\em Ann. Probab.}, 15(3):1052--1061, 1987.

\bibitem{madaule2017}
T.~Madaule.
\newblock Convergence in law for the branching random walk seen from its tip.
\newblock {\em J. Theoret. Probab.}, 30(1):27--63, 2017.

\bibitem{Mallein2018}
B.~Mallein.
\newblock Genealogy of the extremal process of the branching random walk.
\newblock {\em ALEA Lat. Am. J. Probab. Math. Stat.}, 15(2):no. 39, 1065--1087,
  2018.

\bibitem{mckean75}
H.~P. McKean.
\newblock Application of {B}rownian motion to the equation of
  {K}olmogorov-{P}etrovskii-{P}iskunov.
\newblock {\em Comm. Pure Appl. Math.}, 28(3):323--331, 1975.

\bibitem{morters_peres_2010}
P.~M{\"o}rters and Y.~Peres.
\newblock {\em Brownian Motion}.
\newblock Cambridge Series in Statistical and Probabilistic Mathematics.
  Cambridge University Press, 2010.

\bibitem{PainZindy21}
M.~Pain and O.~Zindy.
\newblock {Two-temperatures overlap distribution for the 2D discrete Gaussian
  free field}.
\newblock {\em Ann. Inst. H. Poincar\'e Probab. Statist.}, 57(2):685 -- 699,
  2021.

\bibitem{panchenkotalagrand2007-1}
D.~Panchenko and M.~Talagrand.
\newblock On one property of {D}errida-{R}uelle cascades.
\newblock {\em C. R. Math. Acad. Sci. Paris}, 345(11):653--656, 2007.

\bibitem{Rizzo2009}
T.~Rizzo.
\newblock Chaos in mean-field spin-glass models.
\newblock In {\em Spin glasses: statics and dynamics}, volume~62 of {\em Progr.
  Probab.}, pages 143--157. Birkh\"{a}user Verlag, Basel, 2009.

\bibitem{shi2015}
Z.~Shi.
\newblock {\em Branching random walks}, volume 2151 of {\em Lecture Notes in
  Mathematics}.
\newblock Springer, Cham, 2015.
\newblock Lecture notes from the 42nd Probability Summer School held in Saint
  Flour, 2012, {\'E}cole d'{\'E}t{\'e} de Probabilit{\'e}s de Saint-Flour.
  [Saint-Flour Probability Summer School].

\bibitem{talagrand2011}
M.~Talagrand.
\newblock {\em Mean field models for spin glasses. {V}olume {I}: Basic
  Examples}, volume~54 of {\em Ergebnisse der Mathematik und ihrer
  Grenzgebiete. 3. Folge / A Series of Modern Surveys in Mathematics}.
\newblock Springer-Verlag, Berlin, 2011.
\newblock Basic examples.

\end{thebibliography}

\end{document}